\documentclass[USenglish]{article}
\usepackage{fullpage}
\usepackage{amsmath,amsfonts,amsthm,amssymb}
\usepackage{color,xcolor}
\usepackage{ifpdf}
\usepackage{psfrag}
\usepackage{graphicx,graphics}

\usepackage[small]{caption}
\usepackage{subcaption}

\usepackage{xparse}

\usepackage{authblk}

\usepackage{booktabs}
\usepackage{rotating}
\usepackage{multirow}

\usepackage{url}
\usepackage{hyperref}
\usepackage{todonotes}
\usepackage{ulem} 
\newtheorem{theorem}{Theorem}[section]

\newtheorem{remark}[theorem]{Remark}

\newtheorem{ex}[theorem]{Example}

\usepackage{cancel}

\newcommand{\R}{\mathbb R}

\newcommand{\PP}{\mathbb P}

\renewcommand{\SS}{\mathcal{S}}

\newcommand{\EE}{\mathcal{E}}

\newcommand{\diff}[1]{{\mathrm{d}{#1}}}

\newcommand{\ResKs}{{{\Phi}^K_\sigma}}
\newcommand{\ResK}{{{\Phi}^K}}
        
\newcommand{\bbf}{{\mathbf {f}}}
\newcommand{\bxx}{{\mathbf {x}}}
\newcommand{\bn}{{\mathbf {n}}}

\newcommand{\dpar}[2]{\dfrac{\partial #1}{\partial #2}}

\newcommand{\bF}{\mathbf{F}}
\newcommand{\bFF}{\mathcal{F}}
\newcommand{\bbF}{\mathbf{\mathcal{F}}}

\newcommand{\bbg}{\mathbf{g}}

\newcommand{\ba}{\mathbf{a}}
\newcommand{\bx}{\mathbf{x}}

\newcommand{\bs}{\boldsymbol}
\renewcommand{\div}{\operatorname{div}}
\newcommand{\est}[1]{\left\langle#1\right\rangle}
\newcommand{\bU}{\mathbf{U}}

\newcommand{\dd}{\mathrm{d}}

\newcommand{\mean}[1]{\overline{#1}}

\newcommand{\LL}{\mathcal{L}}

\newcommand{\e}[1]{\mathrm{e^{#1}}}
\newcommand{\VV}{\mathcal{V}}

\newcommand{\gnum}{g^{\operatorname{num}}}

\renewcommand{\vec}[1]{\underline{#1}}
\NewDocumentCommand{\mat}{mo}{%
  \IfValueTF{#2}{%
    \underline{\underline{#1}}{#2}
  }{%
    \underline{\underline{#1}}\,
  }%
}

\definecolor{darkspringgreen}{rgb}{0., 0.55, 0.3}
\definecolor{dartmouthgreen}{rgb}{0.05, 0.5, 0.06}
\definecolor{etonblue}{rgb}{0.59, 0.78, 0.64}
\definecolor{airforceblue}{rgb}{0., 0.4, 0.66}
\definecolor{arylideyellow}{rgb}{0.91, 0.84, 0.42}
\definecolor{emerald}{rgb}{0.31, 0.78, 0.47}
\definecolor{uclagold}{rgb}{1.0, 0.7, 0.0}
\definecolor{cadmiumorange}{rgb}{0.93, 0.53, 0.18}

\hypersetup{
pdftitle={}
pdfauthor={P. \"Offner},
pdfpagemode=UseOutlines,
linkbordercolor=0 0 0,
linkcolor=red,
citecolor=blue,
colorlinks=true,
bookmarks = true
}
\begin{document}
\title{{Analysis of the SBP-SAT Stabilization for Finite Element Methods Part II: Entropy Stability}}

\author[$\dagger$]{R. Abgrall}
\author[$\star$]{J. Nordstr\"om}
\author[$\dagger$]{P. \"Offner\thanks{Corresponding author: P. \"Offner, philipp.oeffner@math.uzh.ch}}
\author[$\ddag$]{S. Tokareva}
\affil[$\dagger$]{Institute of Mathematics,
University of Zurich, Switzerland}
\affil[$\star$]{Department of Mathematics, Computational Mathematics, Link\"oping University, Sweden}
\affil[$\ddag$]{{Applied Mathematics and Plasma Physics Group, Los Alamos National Laboratory, USA}}

\date{
\today}
\maketitle

\begin{abstract}

In the research community, there exists the strong 
belief that a continuous Galerkin scheme is 
notoriously unstable and additional stabilization 
terms have to be added to guarantee stability.
In the first part of the series \cite{abgrall2019analysis},
the application of simultaneous
approximation terms for linear problems is investigated where 
the boundary conditions are imposed weakly.
By applying this technique, the authors
demonstrate that a pure continuous Galerkin scheme is indeed linear stable 
if the boundary conditions are done in the correct way. 
In this work, we extend this investigation to the non-linear
case and focusing on entropy conservation. 
Switching to entropy variables, we will provide 
an estimation on the boundary operators also for non-linear problems to 
guarantee conservation. 
In numerical simulations, we verify our theoretical analysis.

\end{abstract}
\section{Introduction}\label{sec:Introduction}

In the first paper \cite{abgrall2019analysis}, the authors demonstrated that a pure Galerkin scheme is
indeed linear stable if the boundary procedure is done in the correct way. 
Here, the usage of simultaneous  approximation terms (SATs)  and imposing the boundary conditions 
weakly are essential for the investigation. \\
In this work, we extend this study to the non-linear case and focusing on entropy stability. 
In \cite{abgrall2018connection}, the author presented a way to build entropy conservative schemes
by adding correction term to it. The term  works at every degree of freedoms (dofs) and 
compensates   the entropy production/destruction at the dofs but has no influence on the accuracy 
and the performance of the scheme. \\
To build now entropy stable Galerkin schemes, we combine both ideas together. 
We add the correction term to the scheme in the non-linear case and further 
develop  a \textbf{new} SAT boundary procedure 
to guarantee, in total, entropy stability.\\
The proposed boundary operators coincide with the classical ones  used the linear case 
\cite{abgrall2019analysis, fernandez2014review, svard2014review, gassner2013skew, offner2019error})
but further extend the SAT approach to non-linear problems. 
Our boundary approach is not restricted  to the continuous Galerkin method but can also
be applied using other schemes (e.g. discontinuous Galerkin\cite{gassner2013skew}, flux reconstruction \cite{ranocha2016summation} or finite differences \cite{fernandez2014review} ).
Therefore, the paper is organized as follows:\\
We explain the residual distribution scheme and the connection to the continuous 
Galerkin framework. Then, we repeat the concept 
about entropy where we follow Harten's approach \cite{harten1983symmetric}.
Next, we explain how to construct entropy conservative/stable schemes
using entropy correction terms developed in \cite{abgrall2018general}
and further extended and applied in \cite{abgrall2018connection, ranocha2019reinterpretation}.
These terms will be essential in our further studies. 
In the next section \ref{sec:linear}, we shortly summarize the  main results about linear stability of the 
pure Galerkin scheme from \cite{abgrall2019analysis} and extend in the following section \ref{sec:non-linear}
our investigation to the non-linear case. We give a recipe how to construct the boundary operators 
to guarantee entropy stability for the Galerkin scheme. 
In numerical simulations, we validate our analysis. Furthermore, we demonstrate also that 
the correction term can even rescue schemes which are not linear stable by construction.
At the end, we summarize our results and give a short outlook.

\section{Residual Distribution Schemes}

In this section, we shortly introduce/repeat the residual distribution (RD) schemes as they  are also well-known in literature,
see \cite{deconinck2000status, abgrall2006residual, abgrall2017high2, abgrall2018general} and references there in.  
RD provides a unifying framework including some of the up-to-date high order schemes 
like continuous and discontinuous Galerkin methods  and flux reconstruction schemes  
\cite{abgrall2018general, abgrall2018connection}. The selection of approximation /solution space and the definition of the residuals 
specifies the scheme completely  and thus the properties of the considered methods.
 In this paper, our focus lies only on the continuous Galerkin scheme. However, we start by repeating the
 general approach here also to include/present the entropy correction terms as they are originally developed in \cite{abgrall2018general},
 extended and applied in \cite{abgrall2018connection, ranocha2019reinterpretation}. 
 
\subsection{Residual Distribution - Notation and Basic Formulation} \label{subsection:RD}
We are interested in the numerical approximation of a hyperbolic problem 
\begin{equation}\label{eq:conservation_law_general}
 \frac{\partial U}{\partial t}+\div f(U)=0
\end{equation}
with suited initial and boundary conditions. 
Later, we will focus on the boundary condition more precisely, 
but for the explanation of RD this is not important.

RD will be used for the discretisation in space. For sake of simplicity, we explain the RD approach for the steady state problem of \eqref{eq:conservation_law_general} and in the case of a globally continuous approximation.
It is 
\begin{subequations}\label{eq:steady_state}
\begin{equation}\label{eq:steady_state:div}
 \div f(U)=0.
\end{equation}
with a Dirichlet condition on the inflow part of the boundary, to fix ideas:
\begin{equation}\label{eq:steady_state:bc}
u(x)=g(x) \text{ for }x\in \partial \Omega^-=\{ y\in \partial \Omega, \nabla_U f(U(y))\cdot \mathbf{n}<0\}
\end{equation}
\end{subequations}
The domain $\Omega$ is split into subdomains $\Omega_h$ (e.g triangles/quads in two dimensions, tetrahedrons/hex in 3D). 
We denote by $K$ the generic element of the mesh and by $h$ the characteristic mesh size. 
For the boundary elements, we denote them  $\Gamma$. Then, the degrees of freedom $\sigma$ (DoFs) are defined 
in 
each $K$:   we have  a set of linear forms acting on the set
$\PP^k$ of polynomials of degree $k$ such that the linear mapping 
$q\in \PP^k\longmapsto (\sigma_1(q),\cdots, \sigma_{|\sum_K|}(q))$ is one-to-one. The set 
$\SS$ denote the set of degrees of freedom in all elements. 

The solution $U$ will be approximated by some element from the space $\VV^h$
defined by 
\begin{equation}\label{eq:solution_space}
\VV^h:=\bigoplus_{K} \left\{ U^h \in \LL^2(K), U^h|_K \in \PP^k   \right\}.
\end{equation}
A linear combination of basis functions $\varphi_\sigma\in \VV^h$  will be used to describe the numerical solution
\begin{equation}\label{eq:solution_approx}
U^h( x)=\sum_{K\in \Omega_h }\sum_{\sigma \in K} U_{\sigma}^h  \varphi_{\sigma}|_K( x), \quad \forall{ x \in \Omega}
\end{equation}
where the coefficients $U_{\sigma}^h$ must be found by a numerical method. Therefore, the residuals 
comes finally into play. Before we focus now on the RD scheme and the definition, we want to make two small remarks.
\begin{remark}$ $
\begin{enumerate}
\item   Two cases for the  solution space \eqref{eq:solution_space} are
normally considered $\VV= \VV^h\cap C^0(\Omega)$ and $\VV= \VV^h$. 
For the continuity requirement, we need an additional condition on the splitting of the domain $\Omega$
e.g. a conformal triangulation while in the second case the conformity can be dropped, see \cite{abgrall2018general, abgrall2018connection}. In this paper, we will make use of $\VV= \VV^h\cap C^0(\Omega)$.
\item Furthermore, as basis functions we are working either with Lagrange interpolation 
where the degrees of freedom are associated to points in $K$  or   B\'{e}zier polynomials. Both are basis of $\PP^k$ but the the B\'ezier polynomials are a viable option for the unsteady case, 
see \cite{abgrall2017high}.
 We note that for any $K$ 
the following condition will be  fulfilled
\begin{equation*}
\forall x \in K, \qquad \sum_{\sigma \in K} \varphi_\sigma (x)=1. 
\end{equation*}

\end{enumerate}
\end{remark}
 \begin{figure}[!htp]
\centering
  \begin{subfigure}[b]{0.35\textwidth}
    \includegraphics[width=\textwidth]{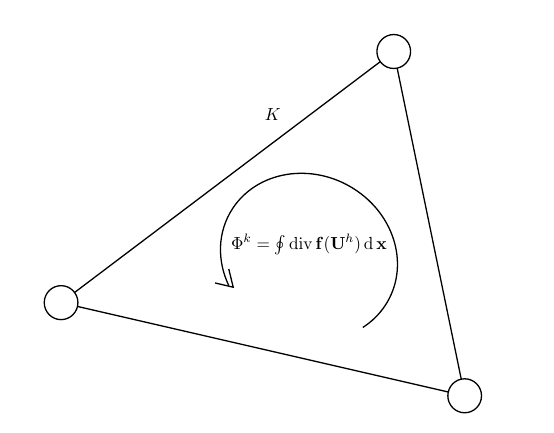}
    \caption{Step 1: Compute the total residual}
  \end{subfigure}%
  ~
  \begin{subfigure}[b]{0.35\textwidth}
    \includegraphics[width=\textwidth]{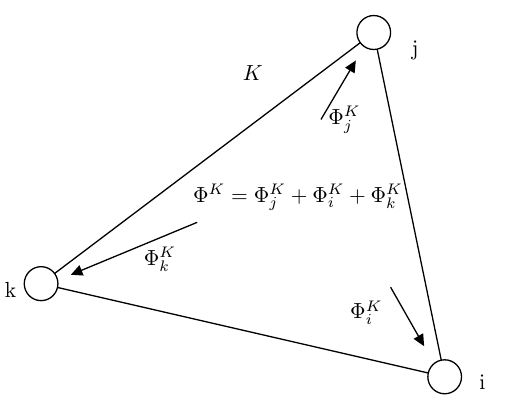}
    \caption{Step 2: Split the total residual}
  \end{subfigure}%
    ~
  \begin{subfigure}[b]{0.3\textwidth}
    \includegraphics[width=\textwidth]{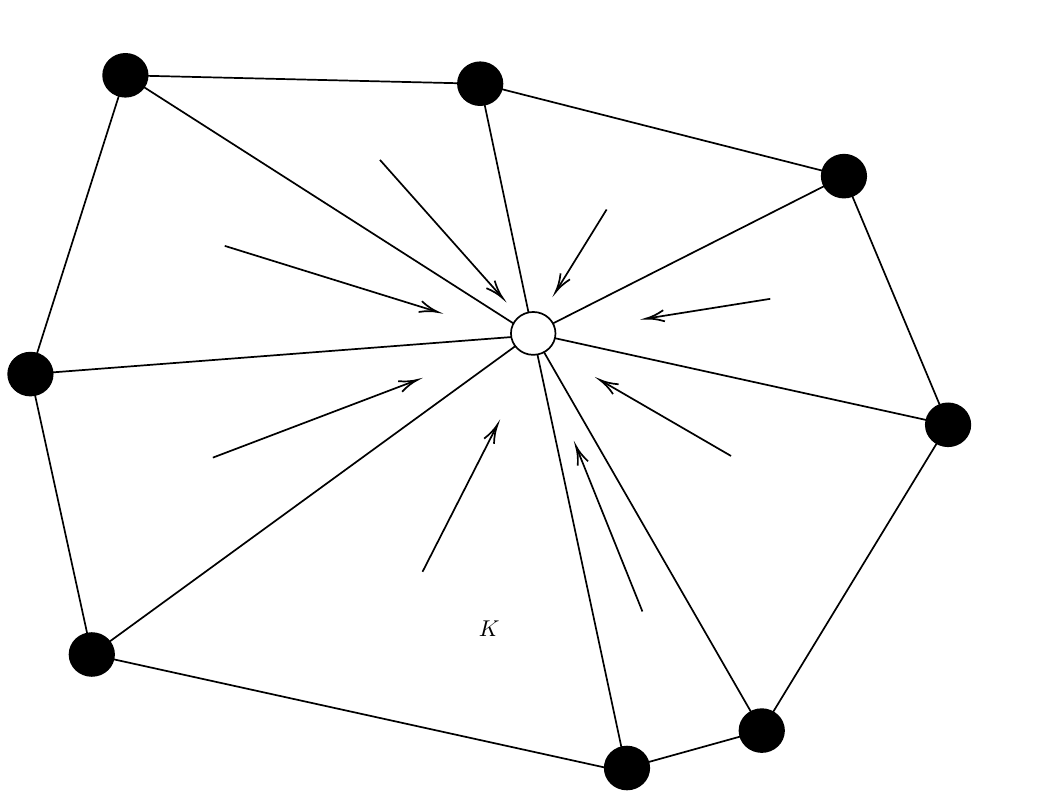}
    \caption{Step 3: Combine the residuals}
  \end{subfigure}%
  \caption{Illustration of the three steps(total residual, nodal residuals, gathering residuals) of the RD approach for linear triangular elements.}
  \label{RD_steps}
\end{figure}

We specify  now the RD scheme by the following three steps:

\begin{enumerate}
\item Define for any $K$ the total residual\footnote{This term is also referred  as fluctuation term in literature \cite{abgrall2003high}.} of \eqref{eq:steady_state}
\begin{equation}\label{eq:fluctuation_term}
\Phi^K=\oint_K \hat{f}_{\mathbf{n}}(U^h, (U^h)^-)\; \diff \gamma 
\end{equation}
where  $\oint$ denotes  a quadrature formula to calculate the integrals, 
$\hat{f}_{\mathbf{n}}$ is a consistent numerical flux in the outward direction $\mathbf{n}$, and $(U^h)^-$ is the generic state on the opposite faces of $K$. In the case of a globally continuous approximation, $\hat{f}_{\mathbf{n}}(U^h, (U^h)^-)=f(U^h)\cdot \mathbf{n}$.
\item Split the total residual into sub-residual $\ResKs$ for each degree of freedom $\sigma \in K$,
so that the sum of all the contributions over an element $K$ is the fluctuation term itself, i.e.
\begin{equation*}
\ResK=\sum_{\sigma \in K} \ResKs, \qquad \forall K \in \Omega_h
\end{equation*}
\item For a face $\Gamma$ on $\partial \Omega$, we also define boundary residuals $\Phi^{\Gamma}_{\sigma}$  which satisfies a 
conservation relation similar to the previous one, namely
\begin{equation*}
\sum\limits_{\sigma\in \Gamma} \Phi^{\Gamma}_{\sigma}=\oint_{\Gamma} \big ( f(U^h)\cdot \mathbf{n}-\hat{f}_\mathbf{n} (U^h, g)\bigg ) d\gamma
\end{equation*} 
\item The resulting scheme if finally obtained by summing all sub-residuals of one degree of freedom from different elements $K$. It is 
\begin{equation}\label{eq:total_scheme}
\sum_{K| \sigma \in K} \ResKs=0  \qquad \forall \sigma \in \Omega_h.
\end{equation}
\eqref{eq:total_scheme} allows us to calculate the coefficients $ U_{\sigma}^h$ in our numerical
 approximation \eqref{eq:solution_approx}.  
 
 In case of $\sigma \in \Gamma$, \eqref{eq:solution_approx} will be split for any (DoF) and we can finally 
write the discretization of \eqref{eq:steady_state}. For any $\sigma \in \SS$, it is
\begin{equation}\label{eq:scheme}
\sum_{K\in \Omega| \sigma \in K}\ResKs(U^h)+
\sum_{\Gamma \in \partial \Omega| \sigma \in \Gamma} \Phi^{\Gamma}_{\sigma}(U^h)=0.
\end{equation}
\end{enumerate}
The RD scheme is described by \eqref{eq:scheme} where the second sum is empty when $\sigma\not\in \partial\Omega$, in which case we get \eqref{eq:total_scheme}.
In the above mentioned literature it is shown that some of the nowadays most favourable schemes 
either based on finite volume  or finite element approaches  (e.g. SUPG, DG, FR, FV-WENO, etc.)
can be recast with this framework \eqref{eq:scheme}. 

With \eqref{eq:scheme} we obtain the RD scheme, what is left over is the exact definition of the residuals $\ResKs$ which will finally specifies the method and the framework we are working in,
the continuous Galerkin scheme. 
Therefore, we chose for the solution space $\VV= \VV^h\cap C^0(\Omega)$ and the residuals $\ResKs$ 
are defined by 
\begin{equation}\label{residual_galerkin}
\begin{aligned}
\ResKs(U^h):=&\oint_{K} \phi_\sigma \div f (U^h)\; \diff x\\
			 =& \oint_{\partial K} \phi_\sigma f (U^h)\cdot \bn\; \diff \gamma 
			 -\oint_K \nabla \phi_\sigma \cdot f(U^h) \;  \diff  x,
\end{aligned}
\end{equation}
 where we used Gauss theorem in the second line. 
 The whole spatial discretisation is now determined. The extension to the unsteady case can be done in a similar way as it is described 
 \textit{inter alia} in \cite{ricchiuto2010explicit,abgrall2017high2}.
 As it is well known that any finite element technique used to semidiscretize \eqref{eq:conservation_law_general} will yield to a formulation 
 \begin{equation}\label{eq:finite_semidiscrete}
  \mat{M}\frac{\partial}{\partial t}\vec{U}^h +\mat{\bF}=0
 \end{equation}
 where $\vec{U}^h$ denotes the vector of degrees of freedom, $\mat{\bF}$ 
 is the approximation of $\div f$ (here: in RD formulation) and 
 $\mat{M}$ is a mass matrix\footnote{In the finite difference community $\mat{M}$ is called norm matrix and is classically abbreviated with $P$, c.f. 
 \cite{svard2014review, nordstrom2006conservative}.}. In case of continuous elements, this matrix is sparse but is not block diagonal, contrarily to the discontinuous
 Galerkin methods. 
 It is well-known that the continuous Galerkin scheme suffers from his stability issues. Therefore, it is common to add  stabilization 
 terms to the scheme as the are introduced for example in \cite{burman2004edge} and which are applied in 
 the RD framework already in \cite{abgrall2006residual, abgrall2017high2, abgrall2018general}. However, we follow a different approach in this paper
 and will renounce these classical stabilisation techniques. 
 In order to do this,  we still need some results known from the literature,
 which we will briefly repeat here.
\subsection{Entropy framework}\label{subsec:entropy}

Since, especially nonlinear conservation laws may have infinity weak solution, one has to 
find additional criterion to select the  \emph{physically relevant} solution among all the weak solutions. 
This criterion is based on the concept of entropy, see \cite{godlewski1991hyperbolic, lax1971shock}. 
In the maravous work \cite{harten1983symmetric}, Harten described the symmetrizability of systems of conservation laws 
which possses an entropy function and repeated also some well-known conditions and
properties of entropy functions.\\
A scalar function  $\eta(U)$ is an entropy function for \eqref{eq:conservation_law_general}
if
\begin{enumerate}
 \item The function $\eta$ satisfies 
\begin{equation}\label{eq:entropy_flux}
\eta'(U)\cdot \nabla_{U}\bbf_j(U)=\nabla_{U} g_j(U),\quad  j=1,\dots d.
\end{equation}
where $g_j$ is some scalar function called entropy flux in the $x_j$ direction.
\item The function $\eta$ is a convex function of $U$. 
\end{enumerate}
$(\eta(U), \bbg(U))$ is the entropy pair with entropy flux $\bbg=(g_1,\dots,g_d)$
and a weak solution of \eqref{eq:conservation_law_general} is called entropy solution if $U$
satisfies additionally 
\begin{equation}\label{eq:entropy_solution}
  \frac{\partial \eta(U)}{\partial t}+\div \bbg(U)\leq0
\end{equation}
for all entropies $\eta$ of \eqref{eq:conservation_law_general}. \\
In the following, we assume to know a entropy function 
of \eqref{eq:conservation_law_general} and 
$V$ is called entropy variable. 
It is given by $V^T:=\eta'(U)=\nabla_{U} \eta(U)$.
Because of the convexity of $\eta$, the mapping 
$U\to V$ is one-to-one
and instead of working with $U$ in \eqref{eq:conservation_law_general} we can work directly with $V$ by setting
$U=U(V)$. Then, by introducing the new variable $V$ in place of $U$ a symmetrization of  \eqref{eq:conservation_law_general} will be accomplished, see  
\cite{harten1983symmetric, mock1980systems} for details.
In abuse of notation we won't introduce a new notation for this change of the variable, but 
we can also rewrite 
\eqref{eq:conservation_law_general} and \eqref{eq:entropy_flux} in terms of the entropy variable.
Also by following \cite{harten1983symmetric, tadmor1987numerical} we know 
the following relations are fulfilled  
where we write everything in terms of the entropy variable $V$
\begin{align}
 \eta'(U)\cdot \nabla_{V}\bbf_j(V)=&\nabla_{V} g_j(V) \Longleftrightarrow 
 V^T\cdot \nabla_{V}\bbf_j(V)=\nabla_{V} g_j (V)   \tag{\ref{eq:entropy_flux}'} \\ 
\bbg(V)=& \est{V,\bbf(V)}-\Theta(V) \label{eq:flux_potential} \\
\nabla_{V} \Theta_j(V)&=\bbf^T_j(V) \label{eq:flux_condition}
\end{align}
where $\Theta=(\Theta_1,\dots,\Theta_d)$ is the flux potential.
For an better understanding, we shortly repeat the main steps of the relations 
\eqref{eq:entropy_flux}-\eqref{eq:flux_condition} from \cite{harten1983symmetric}.
\eqref{eq:entropy_flux} will be derived by using \eqref{eq:flux_potential} 
and \eqref{eq:flux_condition}
\begin{equation*}
 \nabla_{V} g_j(V)=V^T\nabla_{V} \bbf_j(V)+ \bbf_j^T \nabla_{V} V -
 \nabla_{V} \theta_j(V)\nabla_{V} V \stackrel{\eqref{eq:flux_condition}}{=}V^T\nabla_{V} 
 \bbf_j(V).
\end{equation*}
Finally, the concept of entropy has been introduced the consideration
of numerical schemes
\cite{ fisher2013high} and one tries to find/ construct schemes which fulfill a discrete counterpart 
of \eqref{eq:entropy_solution}
which can be written in term of the RD framework 
\begin{equation}\label{ineq:entropy_condition}
\sum_{\sigma\in K }\est{V_\sigma,\ResKs}\geq \oint_{\partial K} \gnum\left(V_{|K}^h, V_{|K^-}^h \right) \bn\diff s
\end{equation}
 where \eqref{eq:entropy_condition} $\gnum$ represents a consistent numerical entropy flux \cite{abgrall2018general}.
One speaks about entropy conservative schemes if the equality in \eqref{ineq:entropy_condition} holds. 
In case of an inequality, the scheme is called entropy stable. 

\subsection{Entropy Correction Terms}\label{subsec:entropy}

In \cite{abgrall2018general}, the author present an ansatz to build entropy conservative / stable schemes in a general framework.
A correction term is added to the scheme at every degree of freedom $\sigma \in K$ and will ensure 
that the schemes fulfill the discrete entropy condition, but simultaneously has no effect on the conservation relation.
In \cite{abgrall2018connection, ranocha2019reinterpretation}, a re-interpretation and applications of these terms can be found.
However, a pure continuous Galerkin scheme is not entropy stable at all. Therefore, the correction term is added 
to the scheme and we shorty repeat the main idea of his procedure from \cite{abgrall2018general}.
\\
 Therefore, 
$V^h\in \VV$ represents the approximation of the entropy variable $V=\nabla_{U}\eta(U)$ in $\VV$.
As we already mentioned the mapping $U \to V(U)$ is one-to-one and we work directly with $V$. 
We are adding the correction term $\bs r_\sigma$ to our residual $\ResK$ at every degree of freedom.
It is 
\begin{equation*}
\hat{\ResKs} = \ResKs+r_\sigma
\end{equation*}
where the discrete entropy condition 
\begin{equation}\label{eq:entropy_condition}
\sum_{\sigma\in K }\est{V_\sigma, \hat{\ResKs}}= \oint_{\partial K} \gnum\left(V_{|K}^h, V_{|K^-}^h \right) \bn\diff s
\end{equation}
holds. 
In \cite{abgrall2018general}, the following correction term is introduced
\begin{equation}\label{eq:correction}
 r_\sigma =\alpha(V_\sigma -\mean{V} ) \text{ with } \mean{V} = \frac{1}{\#K}  \sum_{\sigma \in K } V_\sigma 
\end{equation}
with $\#K$ is the number of degrees of freedom in $K$, 
\begin{equation*}
\alpha= \frac{\EE}{\sum_{\sigma \in K } \left(V_\sigma  -\mean{V} \right)^2}  \text{ and } \EE:= \oint_{\partial K} \gnum\left(V_{|K}^h, V_{|K^-}^h \right) - \sum_{\sigma\in K }\est{V_\sigma, \ResKs}.
\end{equation*}
Adding  \eqref{eq:correction}  to our residual (here: continuous Galerkin) $\ResK$ will provide
that our constructed scheme $\hat{\ResKs} $ fulfills the entropy condition \eqref{eq:entropy_condition}
and, by construction of $r_\sigma $, the conservation relation  is guaranteed, since 
 \begin{equation*}
\sum_{\sigma \in K}\bs r_\sigma= \sum_{\sigma \in K}\alpha(V_\sigma -\mean{V} )=0
\end{equation*}
is satisfied. The error behavior of $\EE$ can be controlled through the used numerical quadrature, see \cite{abgrall2018general} for details and later.
\begin{remark}
The correction term can be seen also as the approximation of a second order derivation and a further stabilisation term. However, the term works directly 
on the degrees of freedom and guarantee that  the semidiscrete entropy inequality is fulfilled. 
Through the conservation relation and the application of high order quadrature formulas the accuracy of the scheme is not affected.
Actually, these terms are only compensating the entropy production/destruction of the original scheme in the semidiscrete setting. 
\end{remark}

\section{Linear Stability for a pure Galerkin Scheme}\label{sec:linear}

In \cite{abgrall2019analysis}, linear stability was investigated
for a pure Galerkin scheme using 
simultaneous approximation terms (SATs).
The origin of the SAT approach lies in the finite difference community,
together with summation by parts (SBP) operators they are used to prove 
linear stability \cite{nordstrom2016roadmap, svard2014review}. The main idea is to impose the boundary 
conditions weakly where also  boundary operators (penalty terms)
are involved. The operators are determined to guarantee 
linear stability in the semi-discrete setting.\\
The SBP-SAT technique has already been applied in the 
discontinuous Galerkin (DG) and Flux Reconstruction (FR) 
framework to prove stability results, see inter alia \cite{gassner2013skew, gassner2016well,chen2017entropy,chenreview, ranocha2016summation} and references therein.
However, to the best of our knowledge, we were the first ones 
who applied this technique together with a pure Galerkin scheme in \cite{abgrall2019analysis}.
We were able to prove linear stability and to derive conditions on the boundary operators.
Here, we will repeat shortly the main ideas and results from \cite{abgrall2019analysis}
to further extend them to the non-linear case in the next section.

\subsection{SATs in the Galerkin Framework }
For simplicity reasons, we are considering the following linear advection equation 
 \begin{equation}\label{eq:linear_ad}
 \begin{split}
    u_t+au_x&=0, \quad 0\leq x\leq 1, \quad t>0,\\
    u(x,0)&=u_{in}(x),\\
    u(0,t)&=b_0(t)
 \end{split}
 \end{equation}
where $u_{in}$ and $b_0$ are initial and boundary data in $L^2$ and $a>0$.\\
Instead of having an extra equation for the boundary condition as in \eqref{eq:linear_ad},
the condition is enforced weakly by some term $\Pi (u) \delta_{x=0}$ which is called Simultaneous
Approximation Term. 
One considers   
$$u_t+a u_x= \Pi (u) \delta_{x=0}$$
In a Galerkin FE based discretization the solution is approximated by 
$u^h(x, t)=\sum\limits_{j=0}^N u^h_j(t)\varphi_j(x)$ where $\varphi$ are basis functions and $u^h_j$ are the coefficients.
Let us further assume that $\varphi$ are Lagrange polynomials where the degrees of freedoms are associated to points in the interval.  
Consider the variational formulation of \eqref{eq:linear_ad} with test function $\varphi_i$ and  inserting the approximation 
yields 
\begin{equation}\label{eq:approx}
\begin{split}
\est{ u_t^h(t,x), \varphi_i(x) }+ \est{a\partial_x u_x^h(t,x), \varphi_i (x)}&\stackrel{!}{=} 0, \quad \forall i=0, \cdots, N. \\
\text{ i. e} \qquad
 \int_I \sum_{j=0}^N (\partial_t u^h_j(t)) \varphi_j(x)\varphi(x) \diff x+a \int_I \sum_{j=0}^Nu^h_j(t)(\partial_x \varphi_j(x) ) \varphi_i(x) \diff x&=0
 \\
 \text{ or } \qquad \sum_{j=0}^{N} M_{i,j}(\partial_t u^h_j(t) ) + a\sum_{j=0}^{N} Q_{i,j} u_j^h(t)&=0\\
\end{split}
\end{equation}
where 
\begin{equation}\label{eq:definition_norm_matrix}
M_{i,j}= \int_I  \varphi_j(x)\varphi_i(x) \diff x \quad \text{ and } \quad Q_{i,j}= \int_I (\partial_x \varphi_j(x)) \varphi_i(x) \diff x.
\end{equation}
We 
consider now 
\begin{equation}\label{eq:Q_inte}
\begin{split}
Q_{i,j}+Q_{i,j}^T&= \int_I (\partial_x \varphi_j(x)) \varphi_i(x) \diff x + (\partial_x \varphi_i(x)) \varphi_j(x) \diff x
= \int_I \partial_x (  \varphi_j(x) \varphi_i(x) ) \diff x\\
&= \varphi_i(x)\varphi_j(x)|_{0}^1=\varphi_i(1)\varphi_j(1)-\varphi_i(0)\varphi_j(0) \quad \forall i,j =0, \cdots, N
\end{split}
\end{equation}
If we set  on both element boundaries degrees of freedom,  then we obtain  
\begin{equation*}
\varphi_i(1)\varphi_j(1)-\varphi_i(0)\varphi_j(0)=\begin{cases}
									1 &\text{ for } i=j=N,\\
									-1 &\text{ for } i=j=0,\\
									0 &\text { elsewhere}.
									\end{cases}
\end{equation*}
The same holds if a quadrature rule is applied which is sufficiently exact. In this case, it can be shown 
that by imposing the boundary conditions weakly together with a suitable boundary operators the pure Galerkin 
scheme is energy stable, see \cite[Proposition 3.4]{abgrall2019analysis}.\\
The result about stability can further generalized to universal FE semi-discretisations 
\eqref{eq:finite_semidiscrete}.
For a general linear problem (scalar or systems) the formulation \eqref{eq:finite_semidiscrete}
can be written with penalty terms as 
\begin{equation}\label{eq:FD_advection}
\mat{M}\partial_t \vec{\bf U}^h+\mat{Q_1}\mat{A} \vec{\bU}^h= \mat{\Pi} \left( \vec{\bU}^h \right)
\end{equation}
where $ \Pi $ is the boundary operators and the Theorem is formulated as follows:
\begin{theorem}[\cite{abgrall2019analysis}, Theorem 3.5]\label{general}
Applying a general FE semidiscretization \eqref{eq:FD_advection}  together with the SAT approach 
to a linear equation
and let the mass matrix $\mat{M}$  of   \eqref{eq:FD_advection} be symmetric. 
If the boundary operator $\Pi$ together with the discretization $\mat{Q_1}$ can be chosen such that 
\begin{equation}\label{test}
(\Pi+\Pi^T)- \left(\mat{Q_1}\mat{A}+ (\mat{Q_1}\mat{A})^T \right)
\end{equation}
has only non-positive eigenvalues $\mat{\Delta}$, then the scheme is (energy) stable. 
\end{theorem}
Therefore, the determination of the operators $\Pi$ 
is essential to guarantee stability.
One is inspired by the continuous energy analysis as it can be seen in 
\cite{abgrall2019analysis} and also in the following.

\subsection{Estimation of the Boundary Operator}

Here, we present now away to estimate the boundary operators using the continuous energy analysis.
We consider as in \cite{abgrall2019analysis} the symmetric linear
hyperbolic system 
\begin{equation}\label{eq:2}
\begin{aligned}
 \dpar{U}{t}+A\dpar{U}{x}+B\dpar{U}{y}&=0,\quad&&  (x,y)\in \Omega, t>0\\
 M_nU&=G_n &&(x,y)\in\partial  \Omega, t>0
\end{aligned}
\end{equation}
where $A,B\in \R^{m\times n}$ are the Jacobian matrices of the system, the matrix $M_n\in \R^{q\times m}$ and the 
vector $G_n\in R^q$ are known. $q$ is the number of boundary conditions to satisfy. 
 $A,B$ are constant and the system \eqref{eq:2} is symmetrizable, i.e. it 
 exists a symmetric and invertible matrix $P$ such that for any vector $\bn=(n_x,n_y)^T$ the matrix 
\begin{equation*}
B_n=A_nP
\end{equation*}
is symmetric with $A_n=An_x+Bn_y$. 
Using the matrix $P$, one can introduce new variables $V=P^{-1/2}U$. 
The original variable can be expressed as
$U=P^{1/2}V$ and the original system \eqref{eq:2} will become 
\begin{equation}\label{eq:3}
 P^{1/2}\dpar{V}{t}+AP^{1/2}\dpar{V}{x}+BP^{1/2}\dpar{V}{y}=0.
\end{equation}
Multiplying \eqref{eq:3} form the left  by $P^{-1/2}$ we obtain the system 
\begin{equation}\label{eq:4}
\dpar{V}{t}+P^{-1/2}AP^{1/2}\dpar{V}{x}+P^{-1/2}BP^{1/2}\dpar{V}{y}=0.
\end{equation}
which is symmetric since $P^{-1/2}AP^{1/2}n_x+P^{-1/2}BP^{1/2}n_y=P^{-1/2}B_nP^{1/2}$.
Focusing on the boundary treatment of the problem and using the weak formulation, we get for the system \eqref{eq:2}:
\begin{equation}\label{eq:1_weak}
\dpar{U}{t}+A\dpar{U}{x}+B\dpar{U}{y}=\Pi(M_nU-G_n)\delta_{\partial \Omega},
\end{equation}
where $\delta_{\partial \Omega}$ is the boundary indicator function and $\Pi$ is our boundary projection 
operator. The estimation of the boundary operator is driven by the energy balance
for the weak formulation \eqref{eq:1_weak} in  the continuous setting. 
We define the global energy of the solution of 
 \eqref{eq:2} by
\begin{equation}\label{eq:global}
E=\frac{1}{2} \int_{\Omega}V^TU \dd \Omega,
\end{equation}
where we take the symmetrizable of the system \eqref{eq:2} under consideration. 
By multiplying \eqref{eq:1_weak} and integrating over $\Omega$, we obtain 
\begin{equation}\label{eq:energy_weak}
\int_{\Omega}V\left(\dpar{U}{t}+A\dpar{U}{x}+B\dpar{U}{y}\right) \dd \Omega= 
\int_{\partial \Omega}V^T\Pi(M_n-G_n)\dd \gamma.
\end{equation}
We reformulate the left side of \eqref{eq:energy_weak}. 
\begin{equation}\label{left:energy}
\begin{aligned}
&\int_{\Omega}V^T\left(\dpar{U}{t}+A\dpar{U}{x}+B\dpar{U}{y}\right) \dd \Omega
=\int_{\Omega}V^T\dpar{U}{t} \dd \Omega+ \int_{\Omega}V^T \left( A\dpar{U}{x}+B\dpar{U}{y} \right)\dd \Omega\\
=&\frac{\dd}{\dd t}\int_{\Omega}\frac{1}{2} V^TU \dd \Omega +\int_\Omega \left(P^{-1/2} U \right)^T \left(A\dpar{U}{x}+B\dpar{U}{y} \right) \dd \Omega= \frac{\dd E}{\dd t}+ \int_\Omega U^T P^{-1/2} \left(A\dpar{U}{x}+B\dpar{U}{y} \right) \dd \Omega\\
&=\frac{\dd E}{\dd t}+ \int_\Omega \frac{1}{2} V^T \left(A n_x+Bn_y \right)U \dd \gamma
=\frac{\dd E}{\dd t}+ \int_{\partial \Omega} \frac{1}{2} V^T A_nU \dd \gamma
\end{aligned}
\end{equation}
Combining \eqref{eq:energy_weak} and \eqref{left:energy}, we get the following energy balance:
\begin{equation}\label{eq:energy_end}
\frac{\dd E}{\dd t}+\int_{\partial \Omega} \left[ \frac{1}{2} V^T A_nU-V^T\Pi(M_nU-G_n) \right] \dd \gamma =0.
\end{equation}
For stability, the energy does not increase, i.e. $\frac{\dd E}{\dd t}  \geq 0$
which is guaranteed if the integral term of the left-hand side of \eqref{eq:energy_end} is non-negative. 
Therefore, 
\begin{equation}\label{eq:boundary}
\frac{1}{2} V^T A_nU-V^T\Pi(M_nU-G_n)\geq 0
\end{equation}
has to be fulfilled.  Inequality \eqref{eq:boundary} imposes the restrictions on the choice of the projection operator 
$\Pi$. Switching to the discrete Finite Element setting, the information of \eqref{eq:boundary} is used to determine 
$\Pi$ and lead to stable pure Galerkin schemes as it is considered in \cite{abgrall2019analysis}.

\section{Investigation of the Boundary Operator - Entropy Stability}\label{sec:non-linear}

As we have seen in section \ref{sec:linear} the key 
to guarantee energy stability using the SAT approach is a proper definition of the boundary operators.
Up-to-now exact integration is always considered in this context.
Therefore, the continuous setting is applied to estimate the boundary operators for linear 
problems using the energy method and the results transform directly from the continuous to the 
discrete framework. \\
Here, we focus on the more general approach and 
describe how to estimate the boundary operators for non-linear hyperbolic systems.\\
As we have seen in \ref{sec:linear}, one key is the  symetrizability of the system 
as it is already described in \cite{friedrichs1958symmetric}.
Therefore, the energy approach can be adapted as it is described in \cite{svard2014review}. 
As it was already named before, the symetrizability of a system is already guaranteed if 
a entropy function exists \cite{harten1983symmetric} but then, we can further 
extend the estimations considering in general non-linear conservation laws \eqref{eq:conservation_law_general}.\\
Here, we give now -up to our knowledge-  a first estimation on the boundary operators to extend
stability results. \\
Extension of the linear case which is explained before for systems. Here, we 
consider the nonlinear case. Therefore, 
we have again $U\in \R^m$. For the sake of simplicity, we are considering homogeneous boundary condition $G_n\equiv 0$ and 
we do further assume that the boundary matrix $\bs{M}_n$ is the identity matrix. Therefore, the boundary 
conditions reads
\begin{equation}
 \bs{M}_nU= U=0, \quad \bxx \in \partial \Omega, \; t>0,
\end{equation}
which means nothing else that the incoming waves are set to zero.\\
Later we restrict ourself to two dimensions but for now on we are 
still in $\Omega\subset \R^d$. 
We have the following conservation law
\begin{equation}\tag{\ref{eq:conservation_law_general}}
 \frac{\partial U}{\partial t}+\div \bbf(U)=0
\end{equation}
where $\bbf_j=\begin{pmatrix}
               f_{1,j}\\
               \hdots\\
               f_{m,j}\\
              \end{pmatrix}
$. If we impose the boundary conditions in the weak form by modifying the right side 
of the system \eqref{eq:conservation_law_general}
gets
\begin{equation}
 \label{eq:conservation_law_weak}
  \frac{\partial U}{\partial t}+\div \bbf(U)=\Pi U \delta_{\partial \Omega}
\end{equation}
where $\delta_{\partial \Omega}$ is the boundary indicator function and $\Pi$ is the boundary projection operator which will be specified.\\
In section \ref{subsec:entropy} we already introduced the concept of entropy. Assuming that 
we have an entropy pair  $(\eta(U), \bbg(U))$ where $\eta$ is a convex entropy function and 
$\bbg$ is an entropy flux for \eqref{eq:conservation_law_general}.  
A symmetrization of \eqref{eq:conservation_law_general} will
be accomplished by introducing the new variable $V$ in place of $U$.
$V$ is called entropy variable and is given by $\eta'(U)=\nabla_{U} \eta(U)=V^T$.
The mapping mapping $U\to V$ is one-to-one
and instead of working with $U$ in \eqref{eq:conservation_law_weak} we are working with $V$ by setting
$U=U(V)$. 
In abuse of notation we won't introduce a new notation 
for this change of the variable, but 
we can rewrite 
\eqref{eq:conservation_law_weak} in terms of the entropy variable
and also using this to determine the boundary projection operator $\Pi$. 
Therefore, we get a dependence in $V$ of $\Pi$ on the right side of \eqref{eq:conservation_law_weak}.
Also by following \cite{harten1983symmetric, tadmor1987numerical}, we know 
the following relations are fulfilled 
where we write everything in terms of the entropy variable $V$.
Now, let us consider the weak formulation \eqref{eq:conservation_law_weak} in terms 
of the entropy variable $V$ and multiply directly by $\eta'$ and integrate in space.
We obtain 
\begin{align*}
    \frac{\partial U(V)}{\partial t}+\div \bbf(V)&=\Pi V \delta_{\partial \Omega} \\
  \int_{\Omega} V     \frac{\partial U(V)}{\partial t}\diff{\bx} 
+ \int_{\Omega}V^T \div \bbf(V) \diff{\bx}&= \int_{\Omega} V^T \Pi V \delta_{\partial \Omega} 
\diff{\bx}  
\end{align*}
Using \eqref{eq:entropy_flux}, we get
\begin{equation}
 \frac{\dd}{\dd t}\int_{\Omega}  \eta(V) \diff{\bx} +\int_{\Omega}
\div \bbg(U) \diff{\bx}= \int_{\partial \Omega} V^T \Pi V \diff{s} 
\end{equation}

Using Gau\ss-Green, we can re-write the second term and get 
\begin{equation}\label{eq:energy_g}
\begin{aligned}
   \frac{\dd}{\dd t}\int_{\Omega}  \eta(V) \diff{\bx} +\int_{\partial \Omega} \bbg(V)\bn \diff{s}= \int_{\partial \Omega} V^T \Pi U \diff{s} \\
      \frac{\dd}{\dd t}\int_{\Omega}  \eta(V) \diff{\bx} +\int_{\partial \Omega} \left( 
      \bbg(V)\bn - V^T \Pi V\right) \diff{s}=0.
\end{aligned}
\end{equation}
For stability \eqref{eq:energy_g}, 
we have to determine $\Pi$ in a way that the left side is not positive or 
\begin{align*}
 \bbg \bn-V^T\Pi V\geq 0 \Longleftrightarrow  \bbg\bn\geq V^T\Pi V 
\end{align*}
Therefore, we have a closer look on $\bbg\bn$. This is a scalar function 
and by assuming smoothness we can apply the mean value theorem\footnote{In abuse of notation 
we use $\bbg'$ for the derivative of $\bbg\bn$ and do not apply everywhere the normal vector $\bn$.
Similar also for the relation \eqref{eq:entropy_flux}.}.
It is 
\begin{equation}\label{eq:mean_the_g}
 \bbg(V)\bn -\bbg(0)\bn= \int_0^1 \bbg'(tV+(1-t)0)V\diff t 
\end{equation}
Using \eqref{eq:entropy_flux} by
\begin{equation*}
  \bbg'(tV+(1+t)0)=tV^T  \bbf'(t V)
\end{equation*}
 yields in \eqref{eq:mean_the_g}
\begin{equation}\label{eq:mean_f}
 \begin{aligned}
  \bbg(V)\bn&= \bbg(0)\bn+ \int_0^1 tV^T \bbf'(t V)V\diff t, \\  
  \bbg(V)\bn&= \bbg(0)\bn+ V^T \underbrace{\left(\int_0^1 t\bbf'(tV) \diff t \right)}_{:=\bFF} V . 
 \end{aligned}
\end{equation}
To estimate $\Pi$ we apply \eqref{eq:mean_f} and get
\begin{equation}\label{eq:boundary_estimate}
 \bbg \bn-V^T\Pi V\geq 0 \stackrel{\eqref{eq:mean_f}}{\Longleftrightarrow} \bbg(0)\bn +V^T\left(\bFF-\Pi\right)V\geq 0 
\end{equation}
$\bbg(0)\bn$ is a constant factor whereas we have to calculate the rest 
depending on the $\bFF$.
The term $\bFF$ can be calculated either with a numerical quadrature formula or
exact. We will demonstrate both calculations in section \ref{sec:numerical}.
Nevertheless, two questions automatically rise:
\begin{itemize}
 \item What is the connection between this estimation and the classical SAT approach 
 for linear problems? 
 \item  How does this estimation affect the stability properties of our continuous
 Galerkin method?
\end{itemize}
The first question can be easily answered. This approach directly simplifies to the linear 
case as it is described before using $U$ for the entropy variable $V$. \\
The second question is more difficult. 
As we mentioned in section \ref{sec:linear} the exactness of the quadrature rule
is essential to transform the estimation from the continuous to the discrete setting.
For a non-linear flux functions this is impossible and we get because of this directly stability issues even for scalar 
equations. Here, the correction term introduced in section \ref{subsec:entropy} comes into play. 
It balance on every inner degree of freedom the entropy error. It allows to have exact calculation of this
part like exact integrations are performed and hence the only remaining term is the boundary term. 
However, it further can 
rescue also in the linear case the scheme if for example the quadrature is not exact enough. We will 
present one example also for this case in the following section.\\
However, the correction term guarantees that in the discrete setting the entropy condition is fulfilled and we can prove the following:
\begin{theorem}\label{eq:entropy stability}
A continuous Galerkin semi-discretization \eqref{eq:finite_semidiscrete} for a hyperbolic conservation law \eqref{eq:conservation_law_general} 
together with the entropy correction term \eqref{eq:correction} and the SAT 
 approach with the boundary operator \eqref{eq:mean_f} is entropy stable. 
\end{theorem}
\begin{proof}
The proof is a combination of the above considered calculation together with the properties of the correction term \eqref{eq:correction}. 
Enforcing the boundary condition weakly using the SAT approach yields a pde \eqref{eq:conservation_law_weak}. 
The semi-discretization with a continuous Galerkin approach leads to some stability issues due to the entropy production of the 
space discretization. However, the correction term is constructed in away to delete/ cancel out these. 
Multiplying with 
the approximation of the entropy variable $V^h$ and using the above manipulation with the boundary operations gives us  finally  
an inequality \eqref{ineq:entropy_condition} in the semi-discrete setting. This shows \ref{eq:entropy stability}.
\end{proof}
Let us consider a first example how to construct the boundary operator \eqref{eq:mean_f}. 
\begin{ex}\label{ex:Burgers}
We are considering the scalar Burgers' equation 
 in this context. The flux is given by $f(U)=\frac{U^2}{2}$. The entropy pair is
 $(\eta,g)=(\frac{U^2}{2},\frac{U^3}{3})$ and the entropy variable $\eta'(U)=V=U$.
 Therefore, we do not have to change anything and we can directly estimate $\Pi$ via $\bFF$.
  $\bFF=\int_0^1 Ut^2 \diff t =\frac{1}{3} U$
and so,
\begin{equation}\label{eq:inequality_burger_general}
U\left(\frac{1}{3}U-\Pi\right)U\geq 0\Longleftrightarrow \frac{1}{3}U\geq \Pi 
\end{equation}
at the boundary $\partial \Omega$. 
We will see the result \eqref{eq:inequality_burger_general} even more often in 
some other examples later.
\end{ex}

\begin{remark}\label{re:time_entropy}
 Entropy stability is guaranteed in the semidiscrete setting. However, to construct fully discrete entropy stable continuous Galerkin schemes
 three different ways can be applied but not further investigated here.
 \begin{itemize}
  \item Using the SBP-SAT technique in time \cite{nordstrom2013summation, friedrich2019entropy} 
lead to  an implicit time integration scheme where the stability 
 results transform directly from the semidiscrete setting to the fully discrete ones. 
 \item  To obtain an explicit method the relaxation 
 technique of Ketcheson \cite{ketcheson2019relaxation} can be applied where after each time step a non-linear equation has to be solved.
 \item  Finally, by using modal filters/artificial viscosity in an adaptive way as presented in \cite{glaubitz2016artificial}, fully discrete 
 entropy stable continuous Galerkin  can be constructed. 
 \end{itemize}
\end{remark}

\begin{remark}\label{eq:appendix}
 Different from the investigation in this section, there exists also other possibilities 
 to determine  the boundary operators. Hence,  most if not all them comes with some restriction 
on the problem or the flux function. If one considers for example
a multi-dimensional scalar case, another approach would be the use of a Taylor series expansion in respect 
to $\bbg(V)\bn:=T_g(V)$ where sufficient smoothness is assumed.\\
Another possibility is if one has some symmetric behavior of the flux function
and can bring all the information to the boundary. 
This is the key in the investigations of \cite{friedrichs1958symmetric, ern2006discontinuous, nordstrom2017roadmap} 
for the linear case and for the nonlinear case \cite{nordstrom2018energy} where the incompressible Navier-Stokes equation is considered.\\
Finally, also   an recursive relation for the flux function can be found and used to determine the boundary operators.
However, these approaches are  not part of this investigation. 
\end{remark}

\section{Numerical Examples}\label{sec:numerical}

We demonstrate that a pure Galerkin scheme is entropy stable if the boundary procedure is done 
via the SAT approach which is introduced in section \ref{sec:non-linear}, and the correction terms are used. 
We impose the boundary conditions weakly and use an adequate boundary operator. 
We use the estimation \eqref{eq:boundary_estimate} and develop $\bbF$
using exact integration and numerical quadrature rules in the experiments (nearly no differences can be seen).
If  $\bbF$ is non-negative, we set the value to zero and for negative $\bbF$
we apply this value. \\
As basis functions either Bernstein or Lagrange polynomials of different 
orders (second to fourth order) on triangular meshes will be applied. 
Nearly no difference can be seen by applying either of these two bases. \
The time integration is done via a strong stability preserving Runge-Kutta (RK) method 
with the same order of accuracy as the space discretization, e.g. \cite{gottlieb2011strong} for details.
Our scheme is semi-discrete entropy stable. Nevertheless, if shocks appear oscillations are
seen and stability issues can appear especially for higher orders.
One can increase the parameters in the free penalty terms 
or add additional streamline or diffusion terms to the correction terms
to increase the stability properties (up-to-now we have only used 
correction terms to guarantee entropy conservation). 

Also approaches to build fully discrete entropy stable 
schemes as described in remark \ref{re:time_entropy} can be used
but all of this is not topic of the current investigation. 

\subsection{Linear Equations}

In the first experiments, we will focus on the correction terms and the influence of these terms on
the scheme. The term \eqref{eq:correction} works specifically on every degree of
freedom and is constructed to force the entropy condition \eqref{eq:entropy_condition}.
\subsubsection*{Stabilisation  Term for Linear Advection}
In the first test, we are  studying  the same problem as described in \cite{abgrall2019analysis}
and consider the  linear advection equation
\begin{equation}\label{eq:scalar}
\dpar{U}{t} + \ba(x,y) \cdot \nabla U = 0, \quad (x,y) \in \Omega, \ t > 0, \\
\end{equation}
where $\ba = (a,b)=(1,0)$ is the advection speed and $\Omega=[0,1]^2$ the domain.
The initial condition is given by
\begin{equation*}
U(x,y,0)=\begin{cases}
          \e{-40r^2},\quad \text{ if }  r=\sqrt{(x-x_0)^2-(y-y_0)^2}<0.25, \\
           0, \qquad \text{ otherwise }
          \end{cases}\\
\end{equation*}
It is a small bump with hight one and located  at $(x_0,y_0)=(0.3,0.3)  $.
We consider homogeneous boundary conditions $G_n\equiv0$
we do further assume that the boundary Matrix $M_n$ is the identity matrix. 
The boundary conditions reads $M_nU=U=0$ for $(x,y)\in \partial \Omega, \;t >0$
 which means nothing else that the incoming waves are set to zero.\\
The bump is moving to the right with speed $1$ and no vertical movement
can be seen. 
By using a pure Galerkin scheme and the SAT boundary procedure 
it can be shown that the scheme is stable under the restriction 
that the used quadrature formula for the mass matrix is exact, c.f. \cite{abgrall2019analysis}.
If we lower the accuracy of the applied quadrature rule the mass matrix is not exact 
anymore (up to machine precision). We touch only the highest order and 
so the mass matrix $M$ in \eqref{eq:finite_semidiscrete} is not exact 
whereas the discretization describes $F$ exact. \\
Even with a low CFL number  $CFL=0.01$ we run into stability issues
and the scheme crashes as it is shown in  in Figure \ref{linear_advection_crash}.
In figure \ref{fig:crash_2}, the structure 
of the bump can still be seen, but, simultaneously, 
the  minimum value is $\approx -2.996$ and the maximum value is around $2.7$.
By going further in time it is getting worth. 
 \begin{figure}[!htp]
 \centering 
   \begin{subfigure}[b]{0.3\textwidth}
    \includegraphics[width=\textwidth]{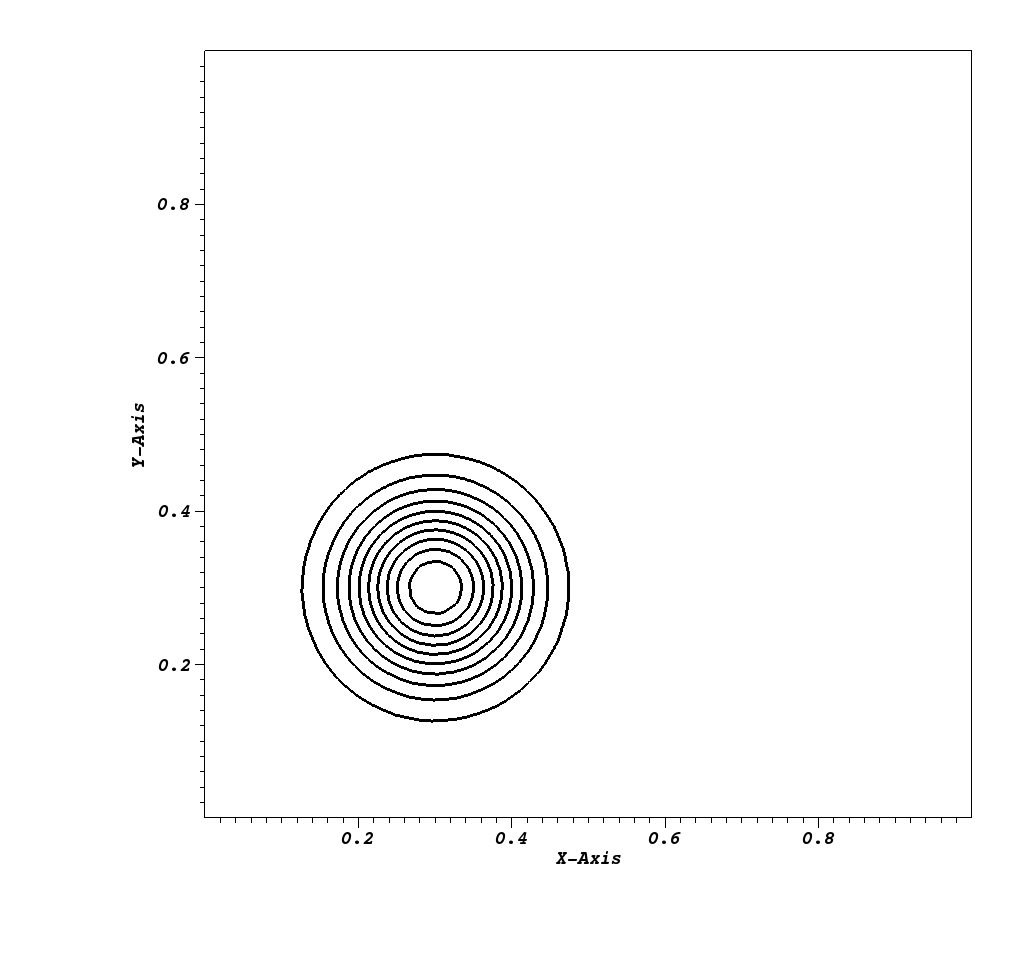}
    \caption{Initial data}
  \end{subfigure}%
   \begin{subfigure}[b]{0.3\textwidth}
    \includegraphics[width=\textwidth]{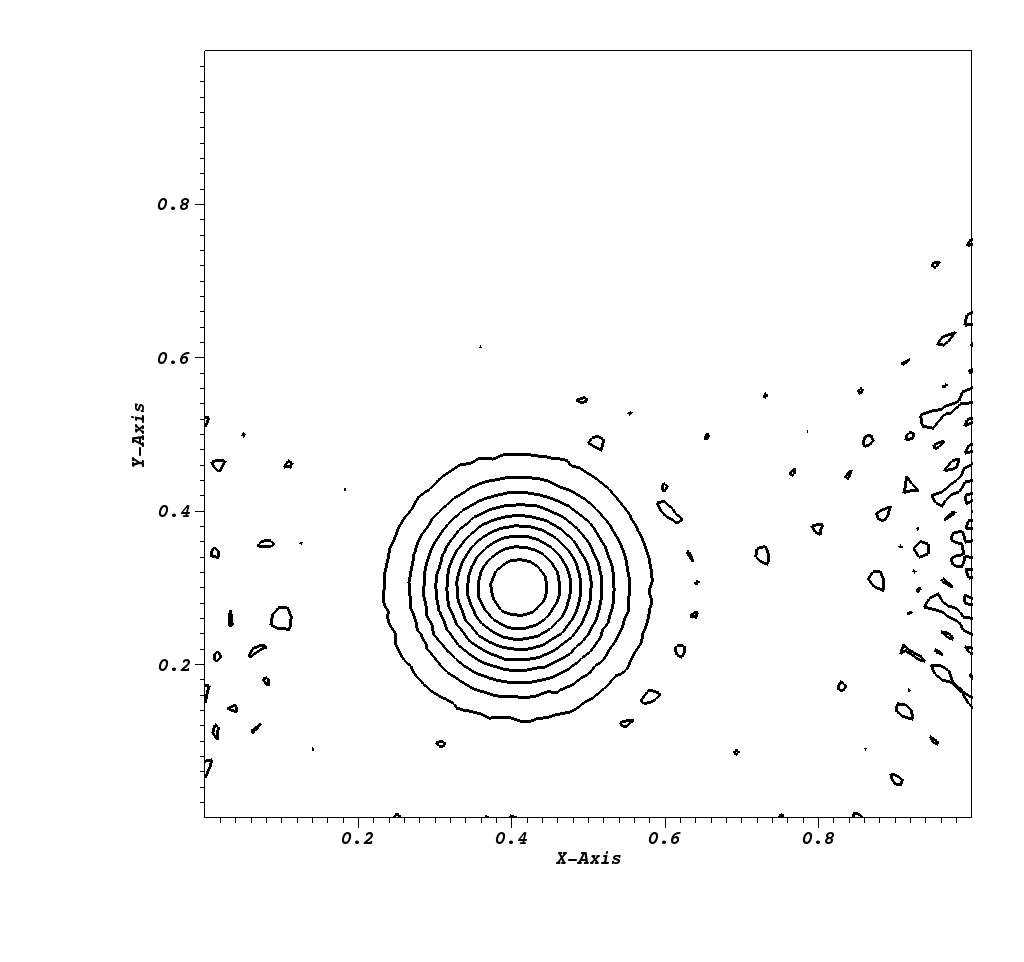}
    \caption{300 steps}
  \end{subfigure}%
     \begin{subfigure}[b]{0.3\textwidth}
    \includegraphics[width=\textwidth]{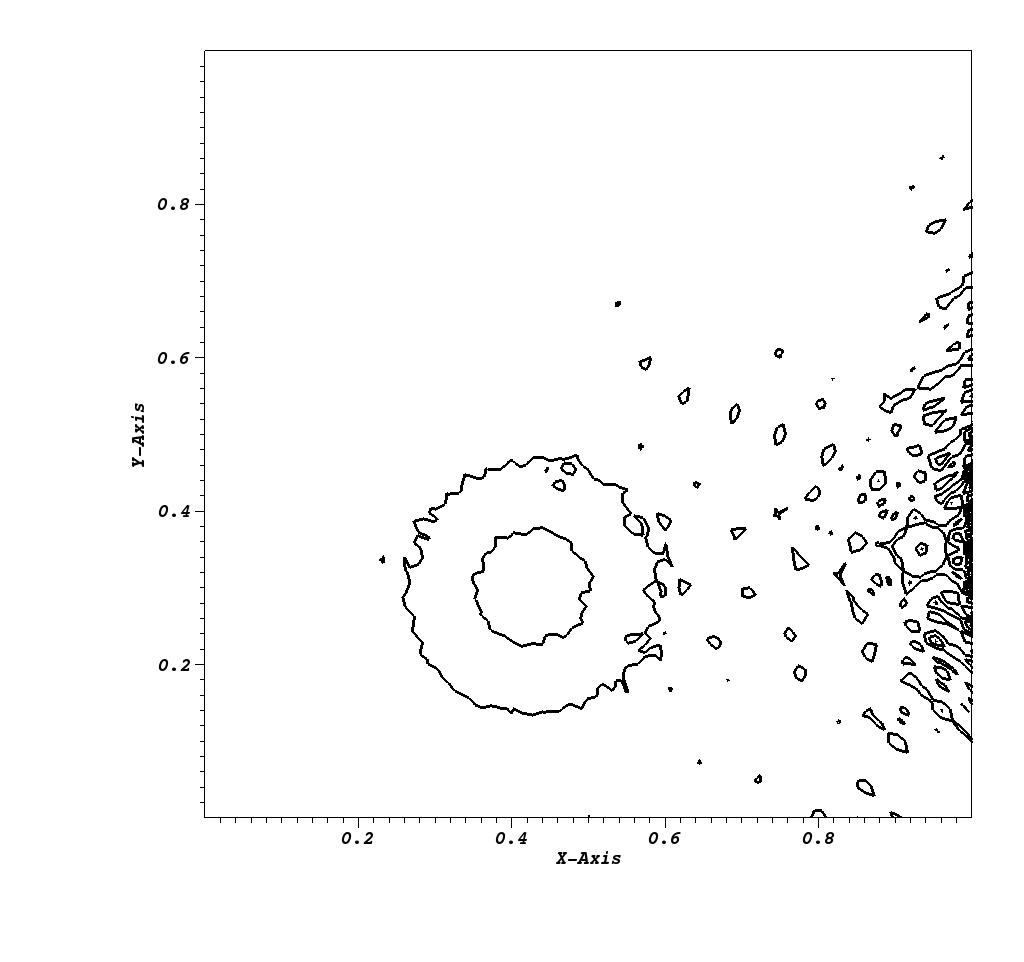}
    \caption{\label{fig:crash_2} 350 steps}
  \end{subfigure}%
     \caption{\label{linear_advection_crash} 4-th order scheme in space and time}
 \end{figure}
By applying the correction term the inaccuracy of the mass matrix can 
be compensated. We run the same test with the correction terms. 
The scheme remain completely stable as it can be seen in Figure
\ref{linear_advection_crash_inexact}.

 \begin{figure}[!htp]
 \centering 
   \begin{subfigure}[b]{0.3\textwidth}
    \includegraphics[width=\textwidth]{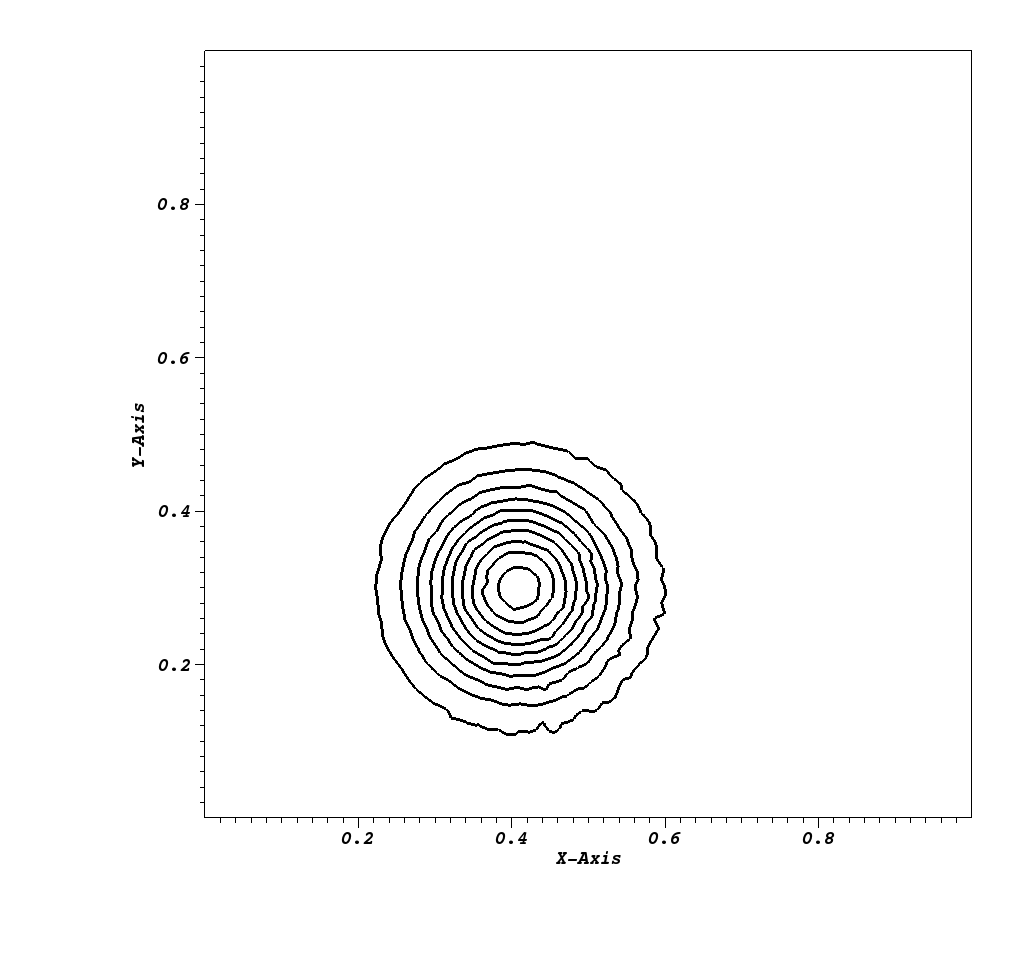}
    \caption{300 steps}
  \end{subfigure}%
     \begin{subfigure}[b]{0.3\textwidth}
    \includegraphics[width=\textwidth]{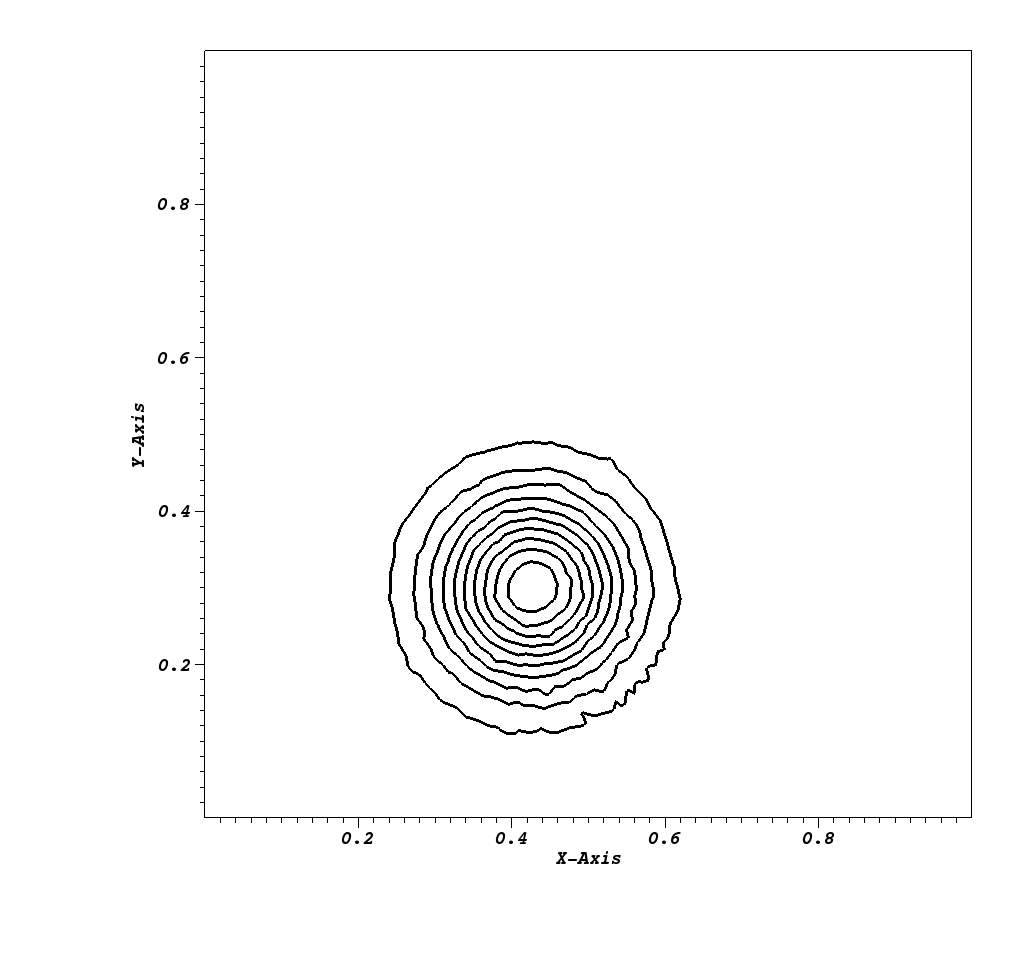}
    \caption{\label{fig:crash} 350 steps}
  \end{subfigure}%
       \begin{subfigure}[b]{0.3\textwidth}
    \includegraphics[width=\textwidth]{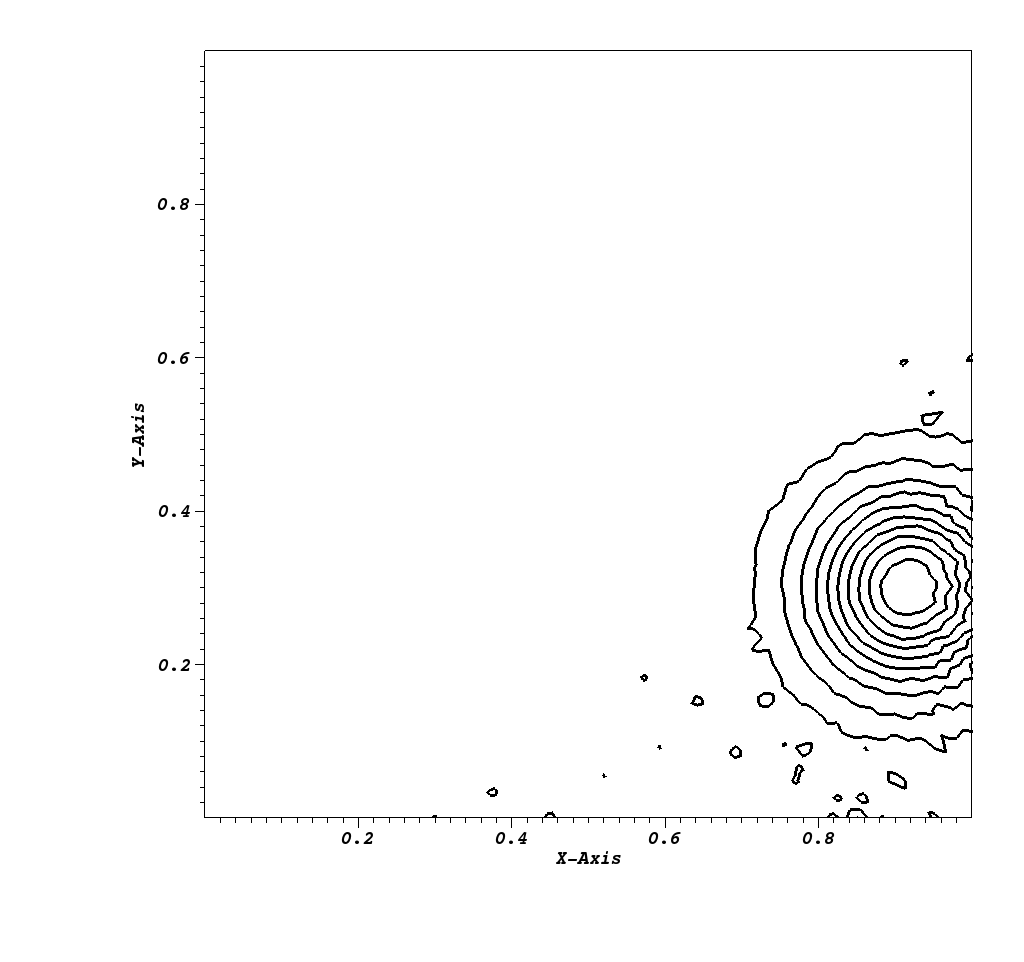}
    \caption{\label{fig:crash} 1700 steps}
  \end{subfigure}%
     \caption{\label{linear_advection_crash_inexact} 4-th order scheme in space and time with correction}
 \end{figure}
 In this test, the correction term works as a classical stabilization 
 term to compensate 
the inaccuracy of the mass matrix. The term works at every degrees of freedom 
and is only applied where it is needed. \\
So, in case of a wrong/imprecise implementation of the Galerkin 
scheme, the term can rescue the stability properties.
Nevertheless, for linear problems it should not be needed if 
the boundary procedure is done via the SAT approach described here 
and more specific in the first part of this series \cite{abgrall2019analysis},
see Figure \ref{linear_advection_1}.

 \begin{figure}[!htp]
 \centering 
   \begin{subfigure}[b]{0.3\textwidth}
    \includegraphics[width=\textwidth]{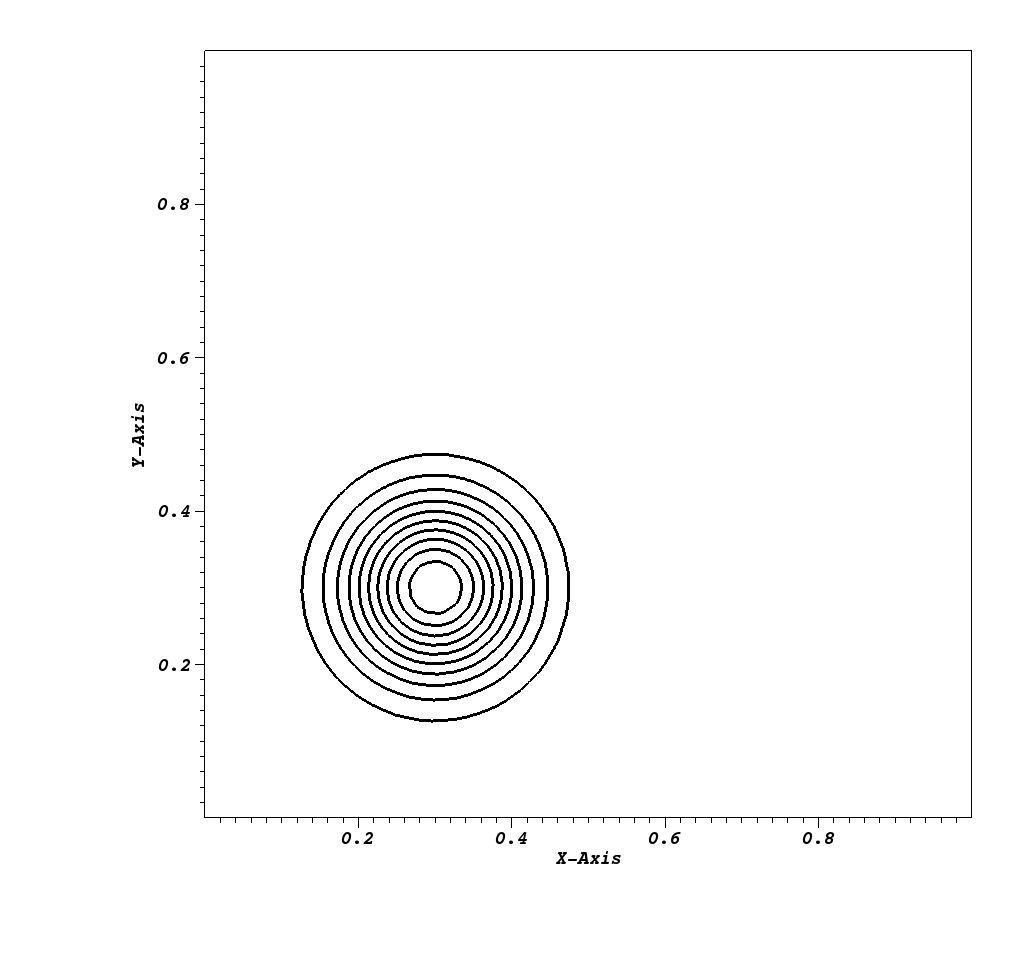}
    \caption{Initial data}
  \end{subfigure}%
   \begin{subfigure}[b]{0.3\textwidth}
    \includegraphics[width=\textwidth]{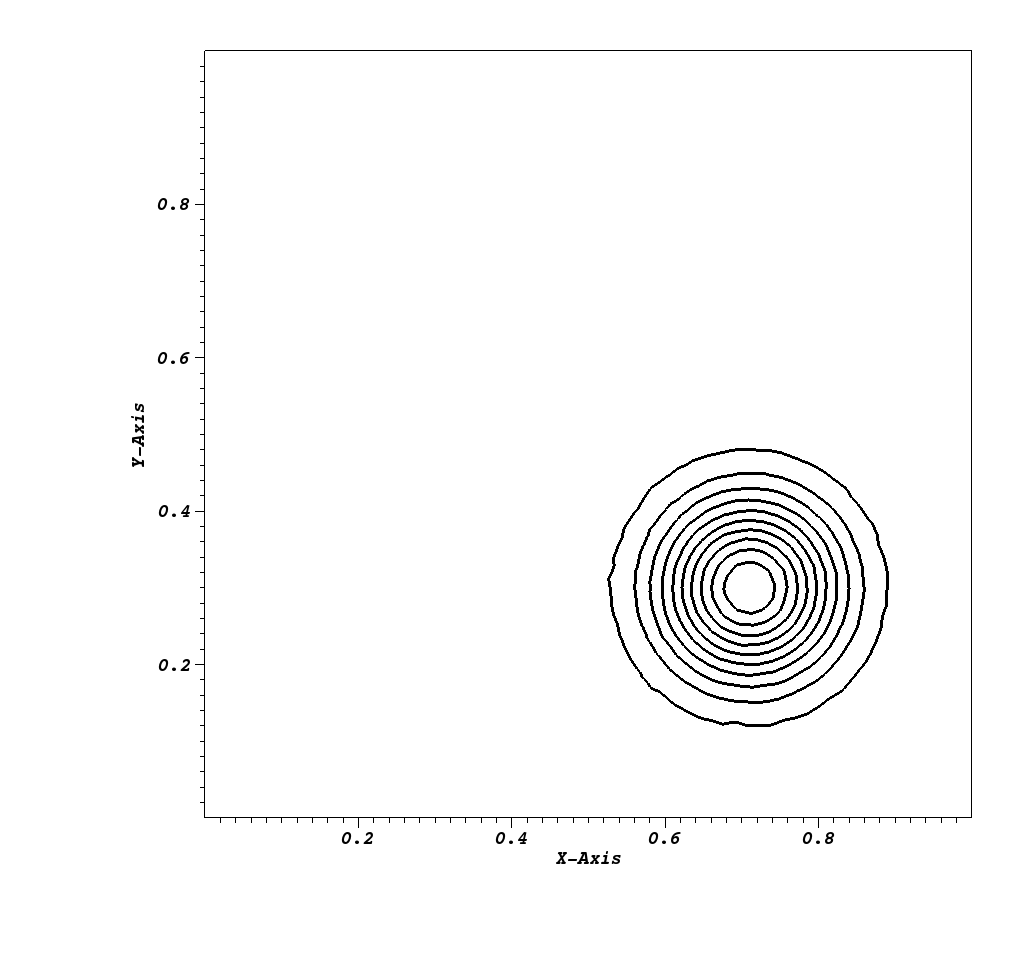}
    \caption{ 100 steps}
  \end{subfigure}%
     \begin{subfigure}[b]{0.3\textwidth}
    \includegraphics[width=\textwidth]{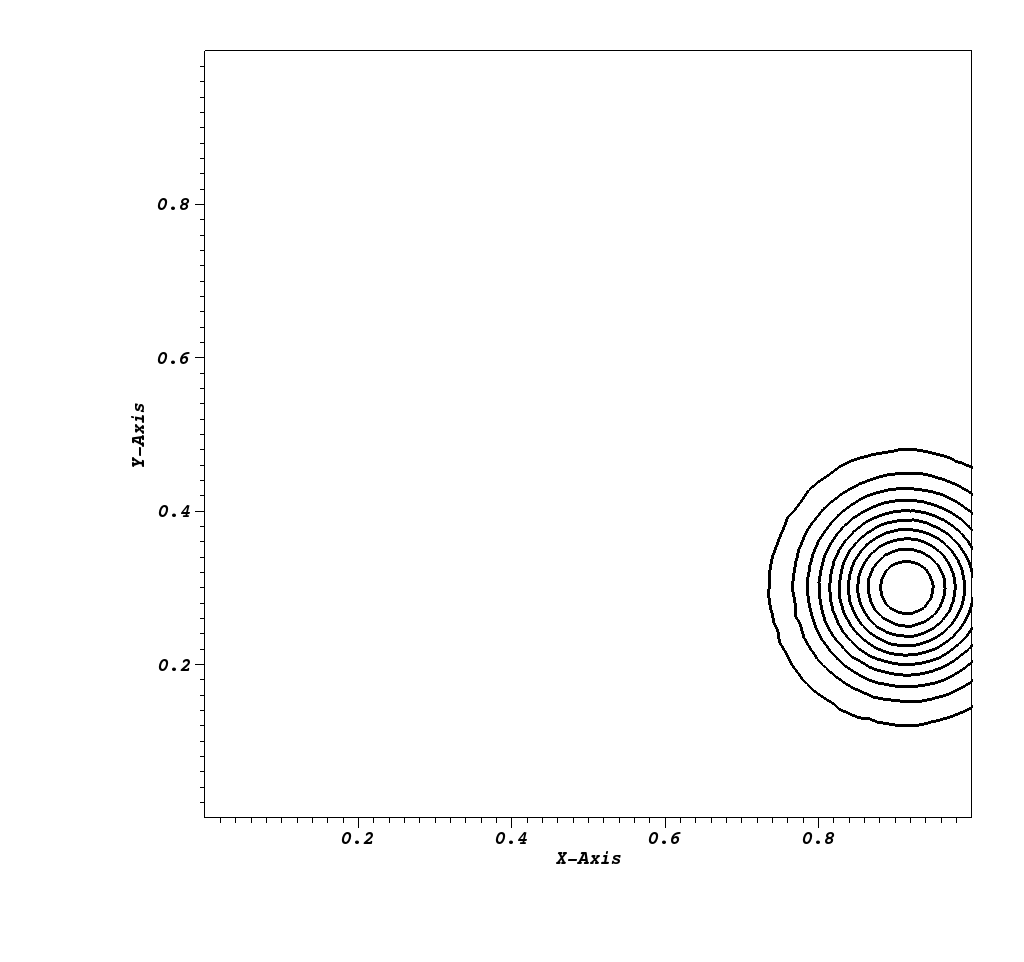}
    \caption{ 150 steps}
  \end{subfigure}%
     \caption{\label{linear_advection_1} 4-th order scheme in space and time, CFL=0.3}
 \end{figure}
 In the next part, we analyze the total amount  
and the influence of the correction terms 
on an exact pure Galerkin scheme. 

\subsubsection*{Influence of the Correction Terms}

The second test is a linear rotation equation in two space dimensions
given by
\begin{equation}\label{eq:linear_advec_two}
\begin{aligned}
      \partial_t u(t,x,y) + \partial_x(2\pi y u(t,x,y)) + \partial_y(2\pi x u(t,x,y)) =0&,
      && (x,y) \in D, t\in (0, 1), \\
     u(0,x,y)=u_0(x,y)=\exp\bigl( -40(x^2+(y-0.5)^2) \bigr)&,&& (x,y)\in D,
\end{aligned}
\end{equation}
where $D$ is the unit disk in $\R^2$.
For the boundary, outflow conditions will be considered and the estimations are done using
the techniques from section \ref{sec:non-linear}. 

For time integration, the third order strong stability preserving scheme SSPRK(3,3) will be used with
CFL number $0.2$ and the correction in each step will be done.
A pure continuous Galerkin scheme of third order is used and Bernstein polynomials
are applied as basis functions.
In this test, a small bump which is located around $(0,0.5)$
is moving around in a circle. The rotation will be finished at $t = 1$.
In figure \eqref{fig:initial} the initial state is presented where 
figure \eqref{fig:mesh_rot} shows the used mesh.
 \begin{figure}[!htp]
 \centering 
   \begin{subfigure}[b]{0.4\textwidth}
    \includegraphics[width=\textwidth]{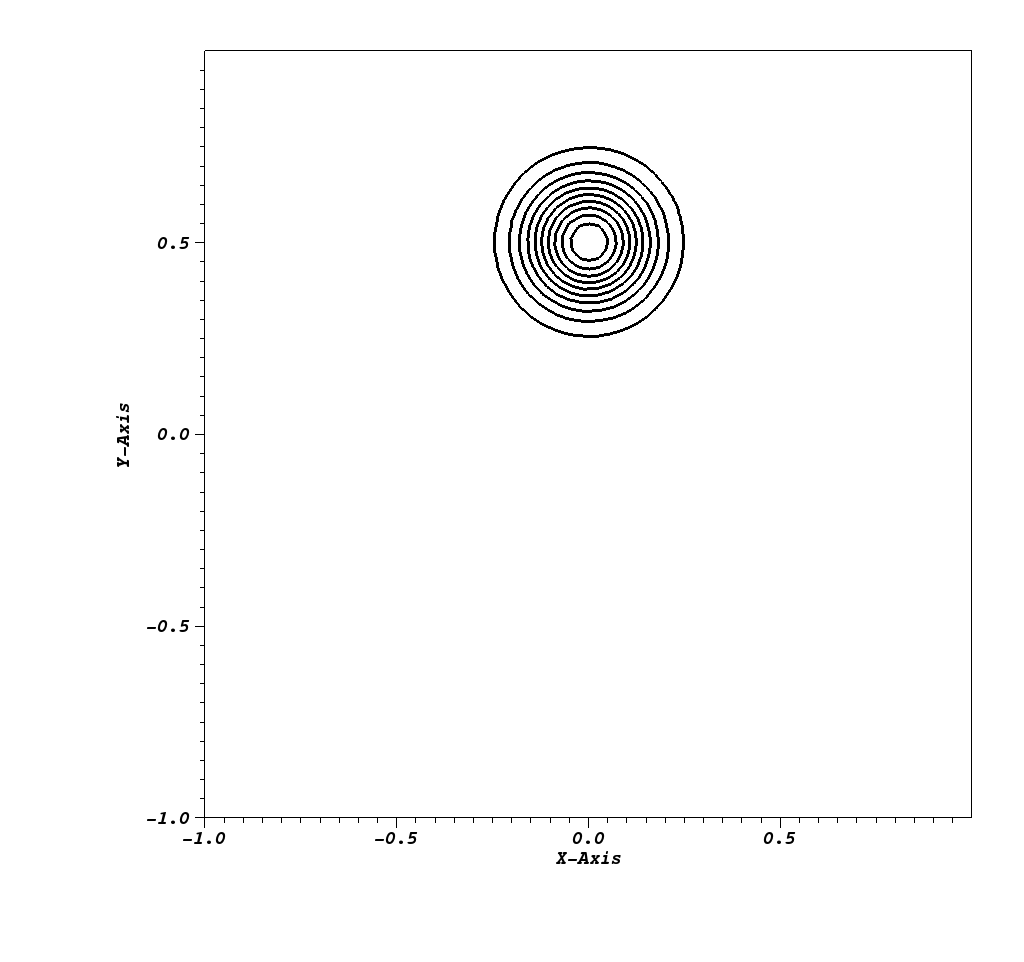}
    \caption{\label{fig:initial} Initial data}
  \end{subfigure}%
   \begin{subfigure}[b]{0.4\textwidth}
    \includegraphics[width=\textwidth]{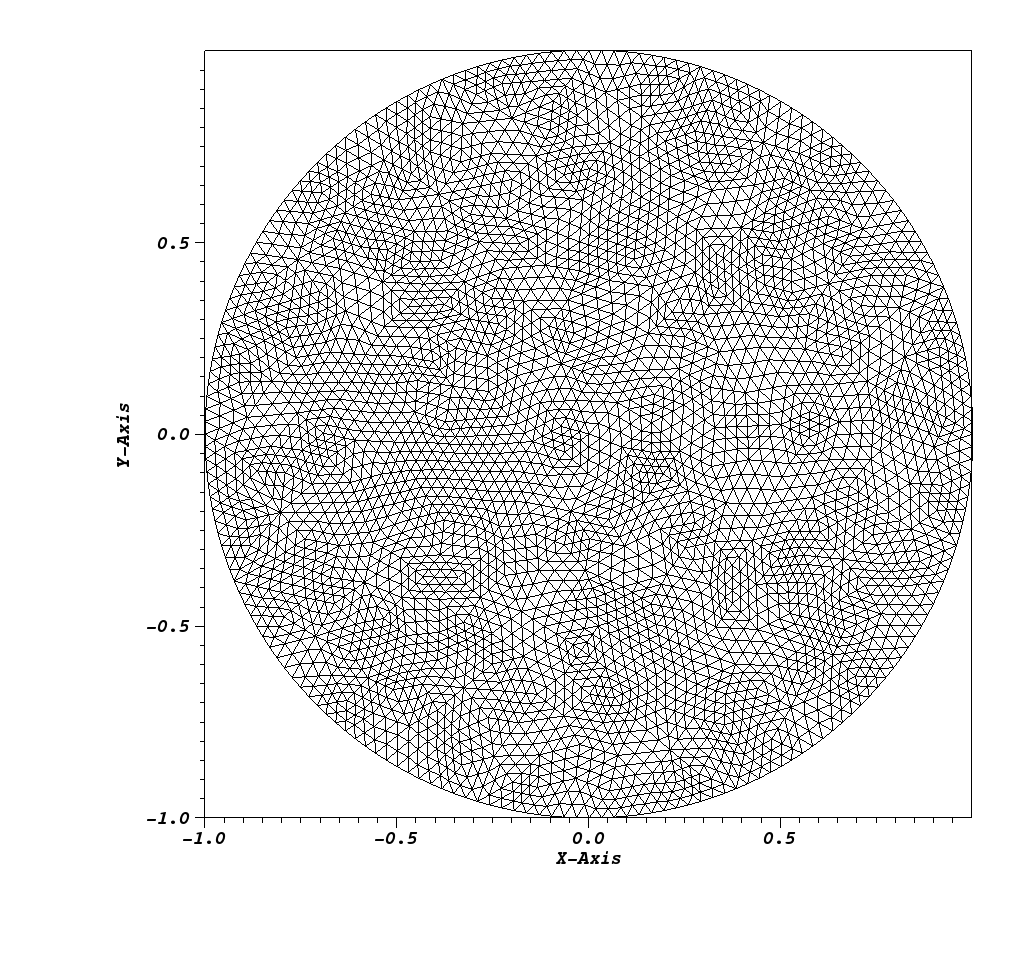}
    \caption{\label{fig:mesh_rot} Mesh}
  \end{subfigure}%
     \caption{4-th order scheme in space and time}
 \end{figure}
The mesh contains 3582 triangular elements. 
We calculate the change of the entropy 
\begin{equation}\label{eq:entropy_change}
    \int_\Omega U^2_{corr}(t)-\int_\Omega U^2_0
\end{equation}
after one rotation (1 second). The value is given by
$6.5020614437081673\cdot 10^{-12}$. Using a finer grid or smaller CFL number
will also lower the entropy production in time. 
Finally, we like to mention 
that we renounce on an error plot here since nearly no differences to
the results presented in \cite{abgrall2019analysis} can be seen. 


 Next,  non-linear equations will be considered.

\subsection{Burgers' Equation}
Here, we focus first on a non-linear problem. We study Burgers' equation 
\begin{equation}\label{eq:linear_advec_two}
\begin{aligned}
      \partial_t u(t,x,y) +\partial_x(0.5 u^2(t,x,y)) + \partial_y (0.5 u^2(t,x,y))=0&,
      && (x,y) \in D, t\in (0, 0.3), \\
     u(0,x,y)=u_0(x,y)=\exp \bigl( -40((x-0.5)^2+(y-0.5)^2) \bigr)&,&& (x,y)\in D.
\end{aligned}
\end{equation}
We use the same mesh and chose also the boundary condition as before. 
The bump is now located in the third quadrant and moves diagonal up. We are considering a non-linear problem and
a shock is formed in time which can 
be recognized in figures \ref{burgers_2}-\ref{burgers_4}.  
The boundary operator $\Pi$ is calculated  using the estimation \eqref{eq:inequality_burger_general}
where we set for non-negative $U$ the boundary operator to  to zero and for negative $U$ we apply 
$\frac{U}{3}$. This procedure guarantees  the inequality  \ref{eq:inequality_burger_general}.
Note that we can also determinate the boundary operator directly calculating 
\eqref{eq:mean_f} using a quadrature rule. Indeed, for most of the quadrature rules this value 
will be exact. 
 \begin{figure}[!htp]
 \centering 
   \begin{subfigure}[b]{0.6\textwidth}
    \includegraphics[width=\textwidth]{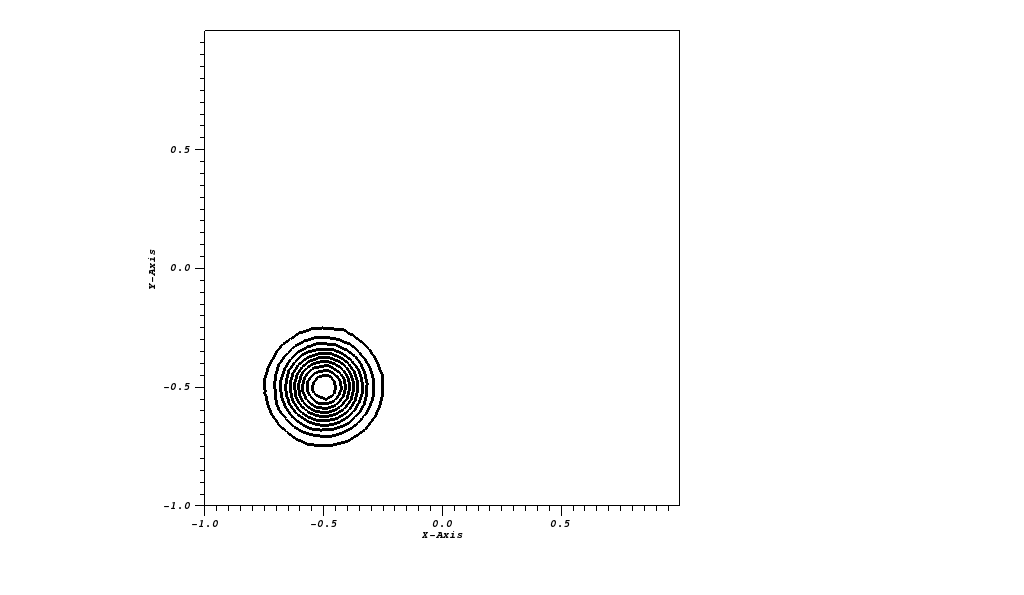}
    \caption{Initial data}
  \end{subfigure}%
   \begin{subfigure}[b]{0.6\textwidth}
    \includegraphics[width=\textwidth]{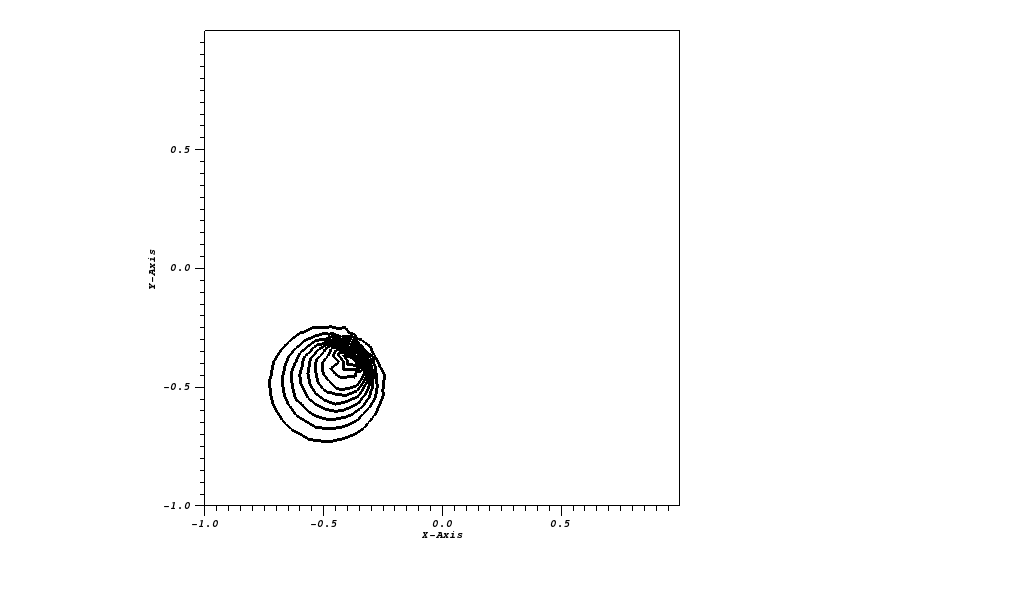}
    \caption{140 steps}
  \end{subfigure}%
  \\
     \begin{subfigure}[b]{0.6\textwidth}
    \includegraphics[width=\textwidth]{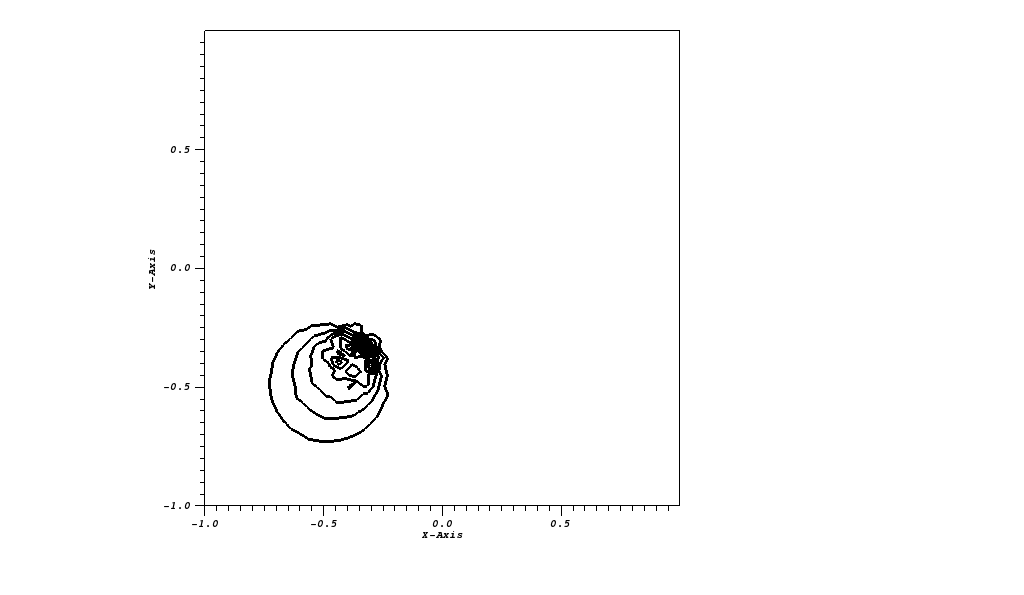}
    \caption{\label{fig:crash} 200 steps}
  \end{subfigure}%
       \begin{subfigure}[b]{0.6\textwidth}
    \includegraphics[width=\textwidth]{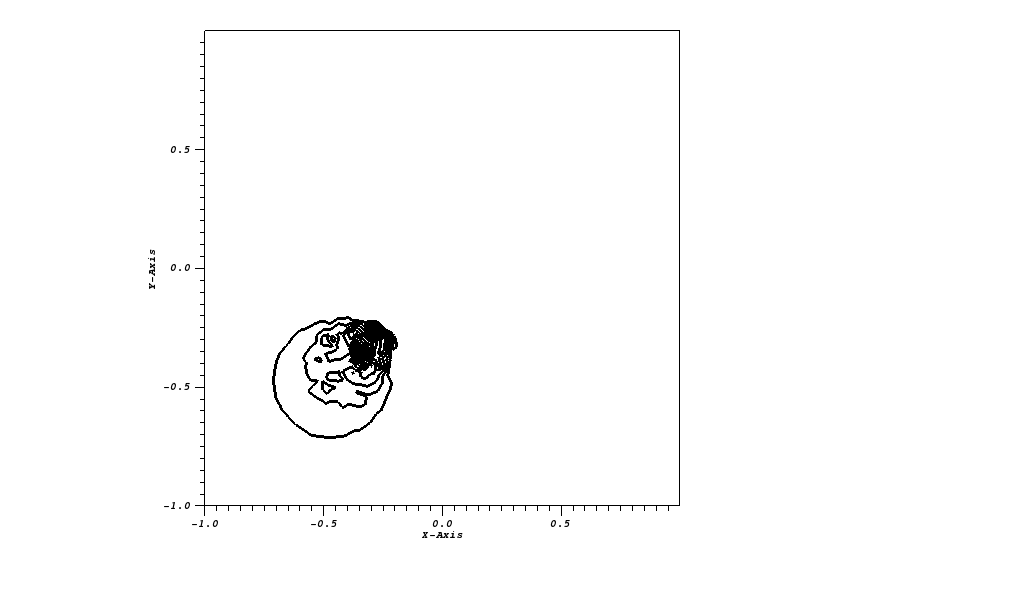}
    \caption{\label{fig:crash} 336 steps}
  \end{subfigure}%
     \caption{\label{burgers_2} 2-th order scheme in space and time}
 \end{figure}

  \begin{figure}[!htp]
 \centering 
   \begin{subfigure}[b]{0.3\textwidth}
    \includegraphics[width=\textwidth]{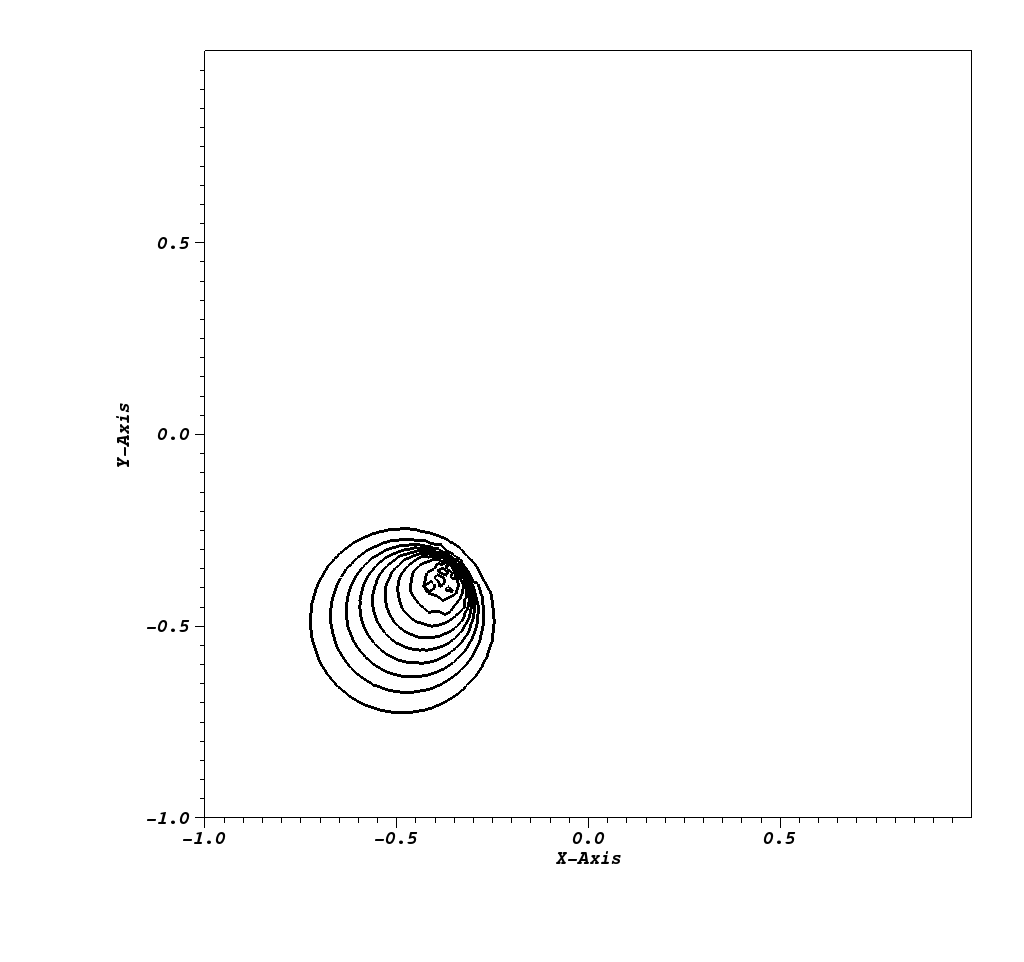}
    \caption{100 steps}
  \end{subfigure}%
     \begin{subfigure}[b]{0.3\textwidth}
    \includegraphics[width=\textwidth]{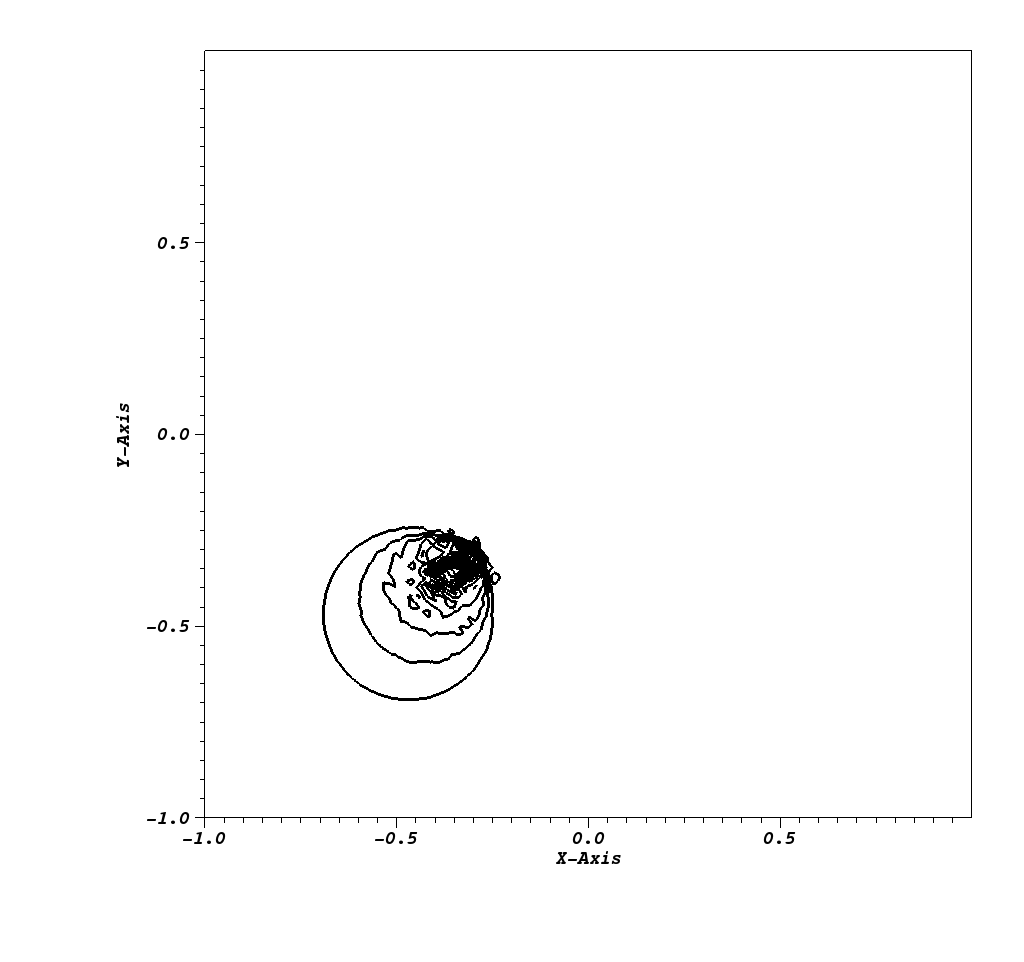}
    \caption{\label{fig:crash} 200 steps}
  \end{subfigure}%
       \begin{subfigure}[b]{0.3\textwidth}
    \includegraphics[width=\textwidth]{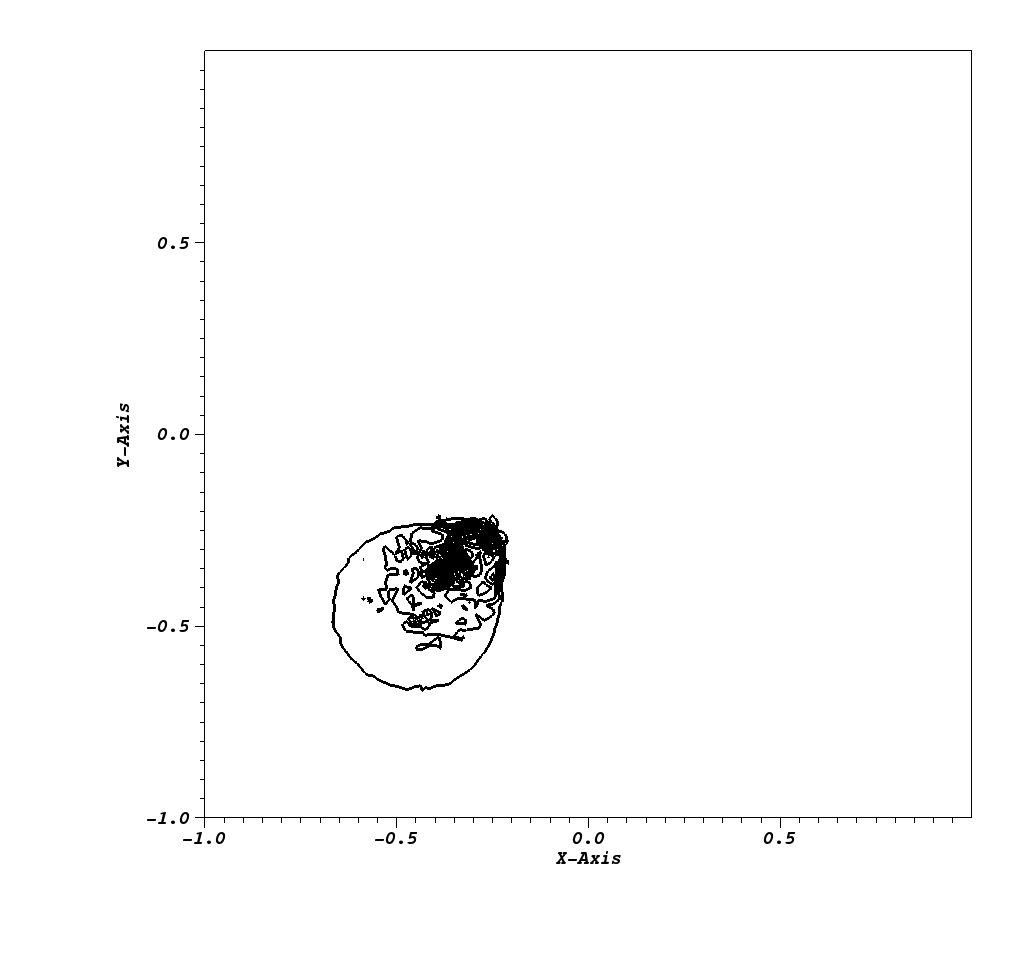}
    \caption{\label{fig:crash} 316 steps}
  \end{subfigure}%
     \caption{\label{burgers_3} 3-th order scheme in space and time}
 \end{figure}

   \begin{figure}[!htp]
 \centering 
   \begin{subfigure}[b]{0.3\textwidth}
    \includegraphics[width=\textwidth]{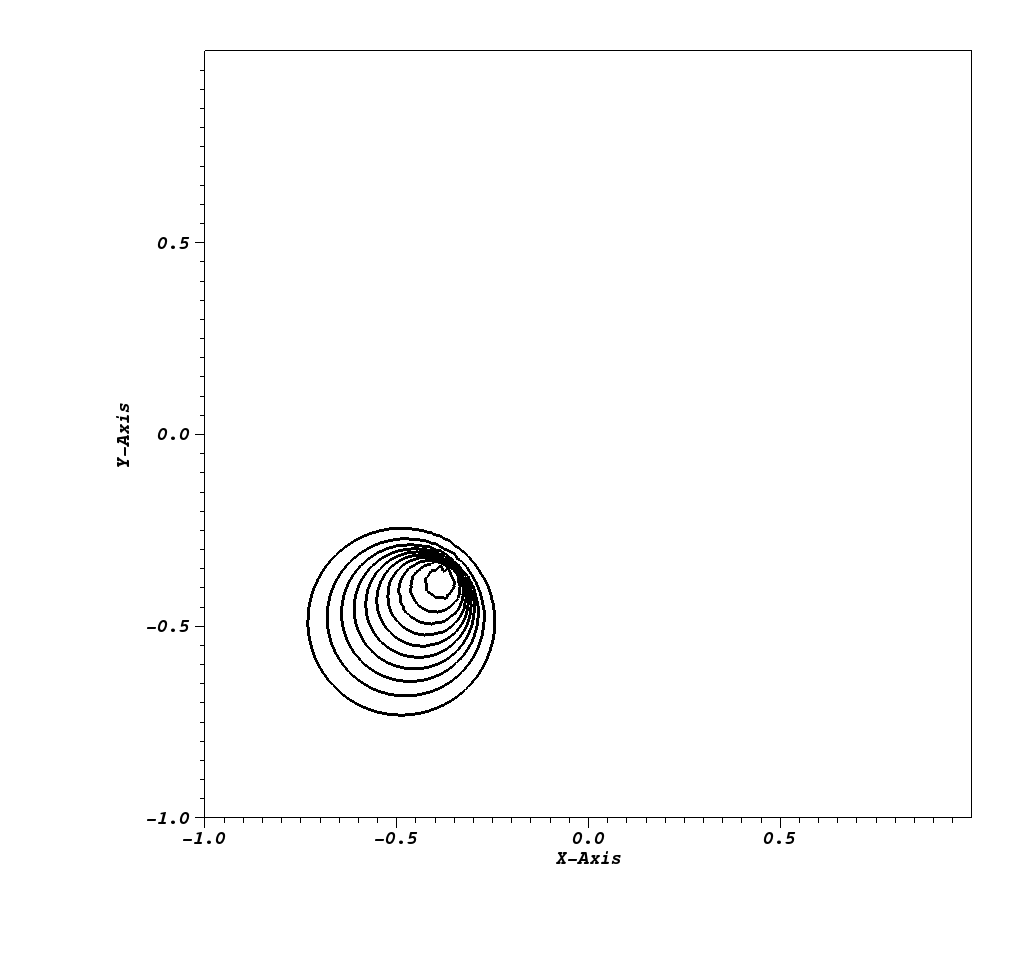}
    \caption{100 steps}
  \end{subfigure}%
     \begin{subfigure}[b]{0.3\textwidth}
    \includegraphics[width=\textwidth]{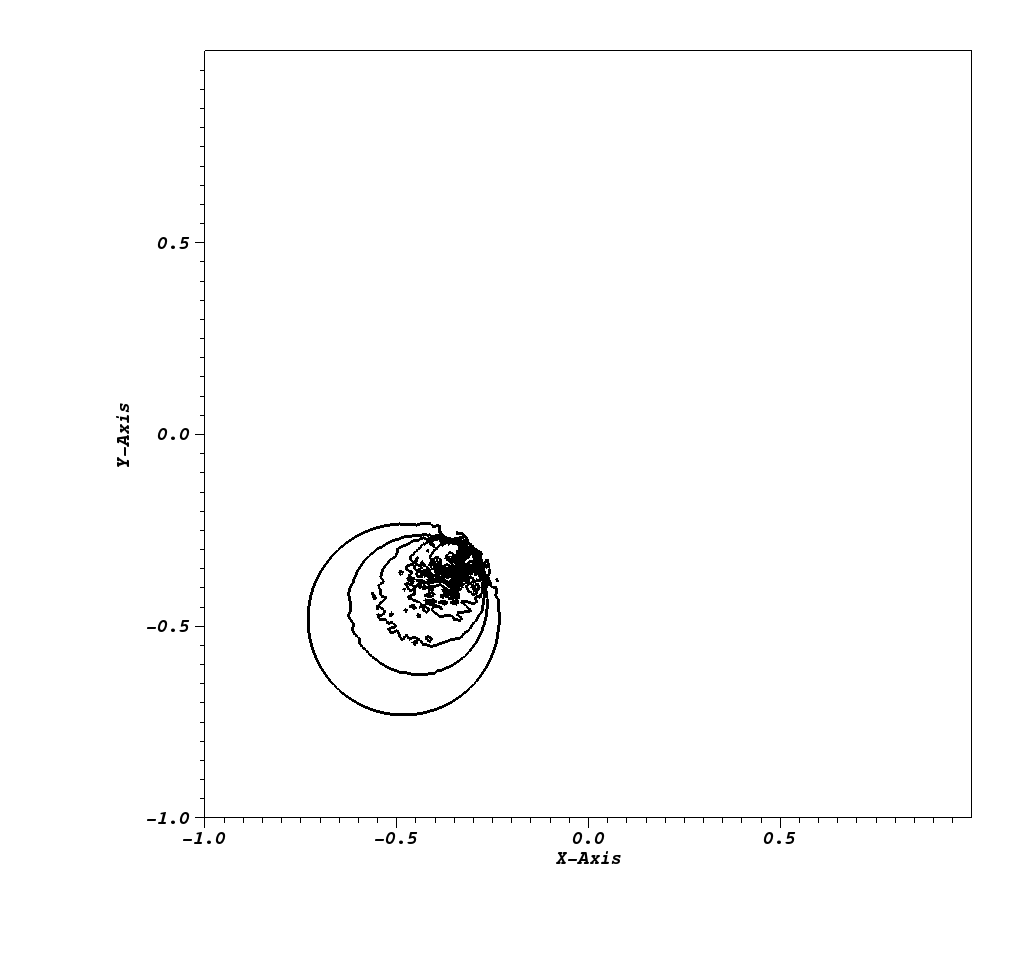}
    \caption{\label{fig:crash} 200 steps}
  \end{subfigure}%
       \begin{subfigure}[b]{0.3\textwidth}
    \includegraphics[width=\textwidth]{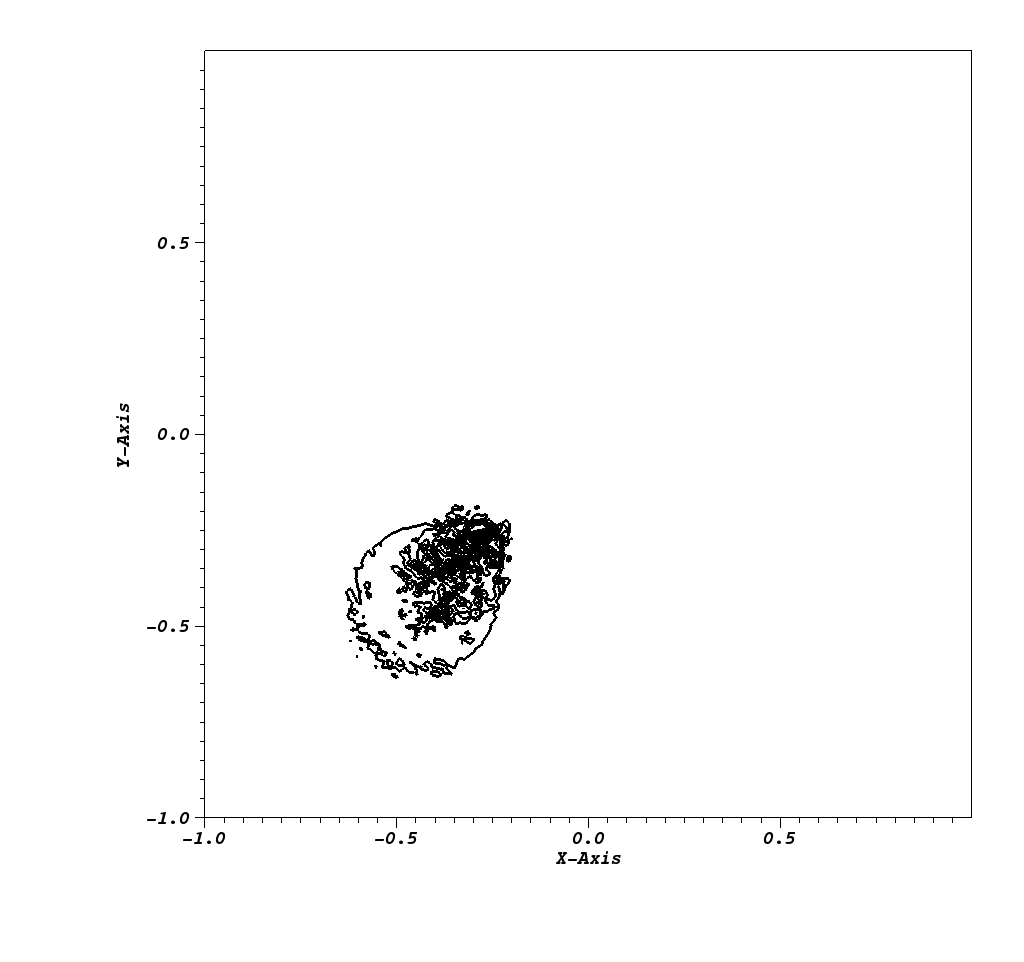}
    \caption{\label{fig:crash} 316 steps}
  \end{subfigure}%
     \caption{\label{burgers_4} 4-th order scheme in space and time}
 \end{figure}

Focusing now on the experiment, the CFL number is set to $0.1$ and in figures  \ref{burgers_2}-\ref{burgers_4}
the experiment is printed using several orders. In the smooth regime the 
fourth order scheme demonstrates the best behavior as it can be seen in figure \ref{burgers_4}
where accuracy is lost after the formation of the shock and oscillations can be seen. 
Also in the second and third order schemes \ref{burgers_2}-   \ref{burgers_3} 
these oscillations appear but not as strong as in the 4th order test case. To delete/handle these oscillations further techniques have to be applied 
but are not topic of these current study.  Next, we consider a Burgers' type of equation where 
the exact evaluation of the boundary operator  using a quadrature formula is not possible anymore.
We investigate the differences and also the change of the entropy. 

\subsection{Burgers' Type of Equation}

The problem is given by
\begin{equation}\label{eq:burgers_equation_like}
\begin{aligned}
      \partial_t u(t,x,y) +\partial_x(\cos u(t,x,y)) + \partial_y u(t,x,y)=0&,
      && (x,y) \in D, t\in (0, 0.2) \\
     u(0,x,y)=u_0(x,y)=\exp\bigl( -40(x^2+y^2) \bigr)&,&& (x,y)\in D,
\end{aligned}
\end{equation}
where $D$ is the unit disk in $\R^2$.
Again, outflow boundary conditions will be considered.
In this test case, a small bump located in zero will move up to the left and a shock
will appear after a finite time. 
First, we consider the total change of the entropy. 
As time integration we choose a SSPRK(3,3) with CFL number $0.1$ and
the space space discretisation is again done by a
pure continuous Galerkin scheme of third order with Bernstein polynomials.
In Table  \ref{entro_table}, the total change of entropy  using the
correction term and the  SAT procedure for the test using 
a grid with $952$ elements. Here, the operator is calculate 
 via a quadrature formula of order five and for non-negative values, it is set to zero again. 
 The calculations ends right before the shock is build. 
 We recognize the total amount is  negative and entropy stability is obtained.
If we increase the accuracy by using a finer grid, the change of the entropy will 
be tend to zero quite fast up to machine precision.

\begin{table}[!ht]
\centering
   \caption{Total entropy change of numerical solutions using a continuous Galerkin
   scheme for the Burgers' type of equation \eqref{eq:burgers_equation_like}.}
   \label{entro_table}
 \begin{tabular}{c|c}
\hline
    Time & Change of entropy \\
\hline
 $  0.0000000000000000 $     & $ 0.0000000000000000     $\\
 $  2.3975956746131837\cdot 10^{-2}$ & $-1.3689936006570046\cdot 10^{-3}$\\
 $  4.7358346029522956\cdot 10^{-2} $& $-1.8770004670986111\cdot 10^{-3}$\\
 $  7.0660969892906739\cdot 10^{-2} $& $-2.0664452188379426\cdot 10^{-3}$\\
 $  9.4077968305402895\cdot 10^{-2} $& $-2.1349903191985441\cdot 10^{-3}$\\
 $ 0.11740611797870178 $ &$     -2.1570788663033569\cdot 10^{-3}$\\
 $ 0.14081576189245676 $ &$     -2.1632360645712908\cdot 10^{-3}$\\
 $ 0.18836214773590210 $ &$     -2.1648782630169544\cdot 10^{-3}$\\
 $ 0.20000000000000001 $ & $      -2.1648948484478503\cdot 10^{-3}$\\
\end{tabular} 
\end{table}
Now, we compare the difference by applying a quadrature rule or exact integration to calculate
$\bbF$ for the estimation of the boundary operator. 

\subsubsection*{Exact Integration}
In Figure \ref{Burgers_type_exact},  a 4th order Galerkin scheme with correction 
term is used. The time integration is done via a strong stability preserving RK method
of fourth order using 5 stages. To estimate $\Pi$ we determine the operator $\bbF$ using
exact integration and apply for $\Pi=\min\{\bbF, 0 \}$. The results are presented in figure 
\ref{Burgers_type_exact}. 
   \begin{figure}[!htp]
 \centering 
   \begin{subfigure}[b]{0.3\textwidth}
    \includegraphics[width=\textwidth]{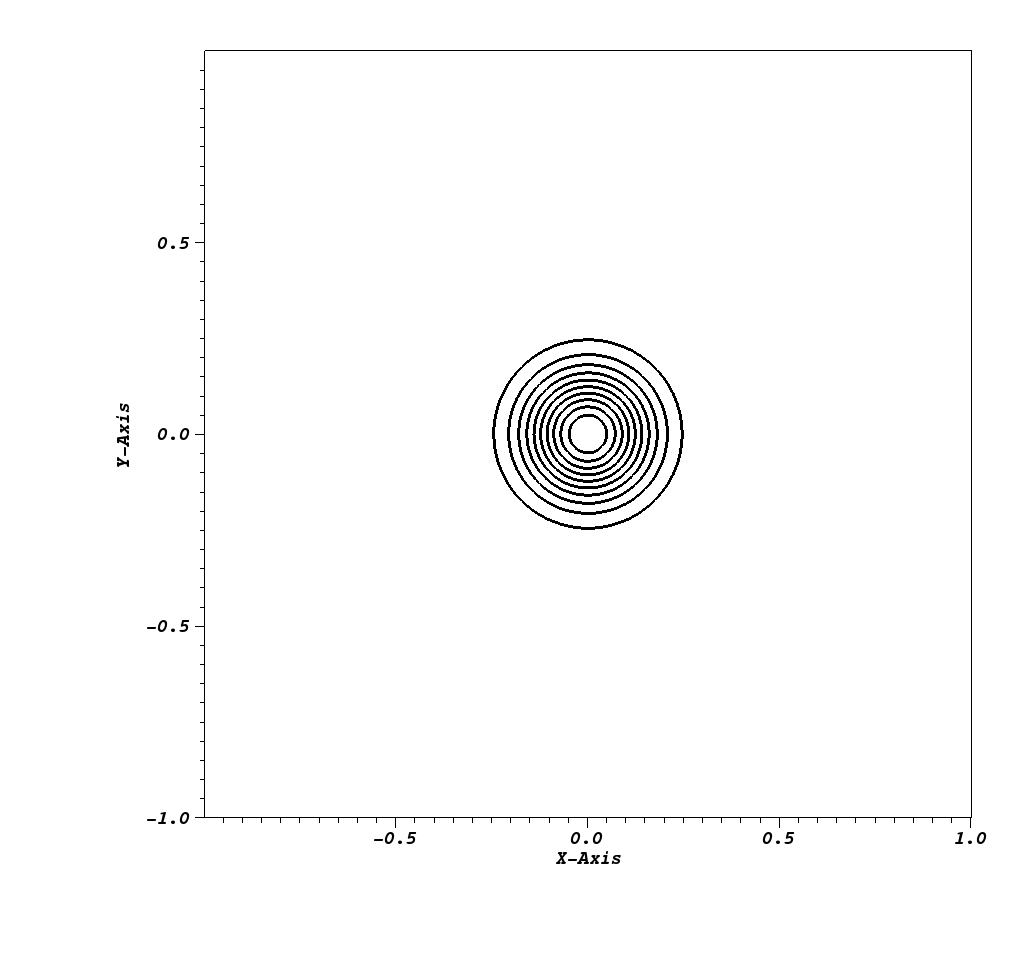}
    \caption{Initial Condition}
  \end{subfigure}%
     \begin{subfigure}[b]{0.3\textwidth}
    \includegraphics[width=\textwidth]{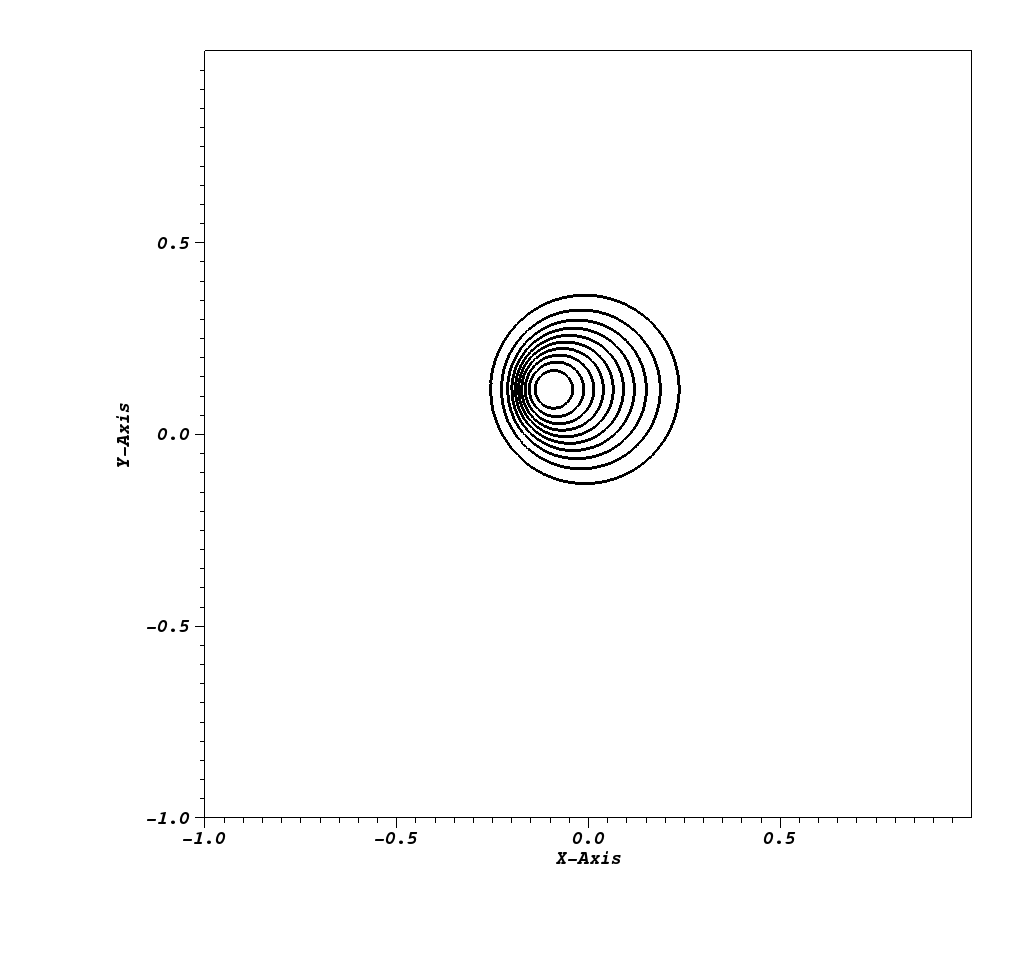}
    \caption{\label{fig:crash} 100 steps}
  \end{subfigure}%
       \begin{subfigure}[b]{0.3\textwidth}
    \includegraphics[width=\textwidth]{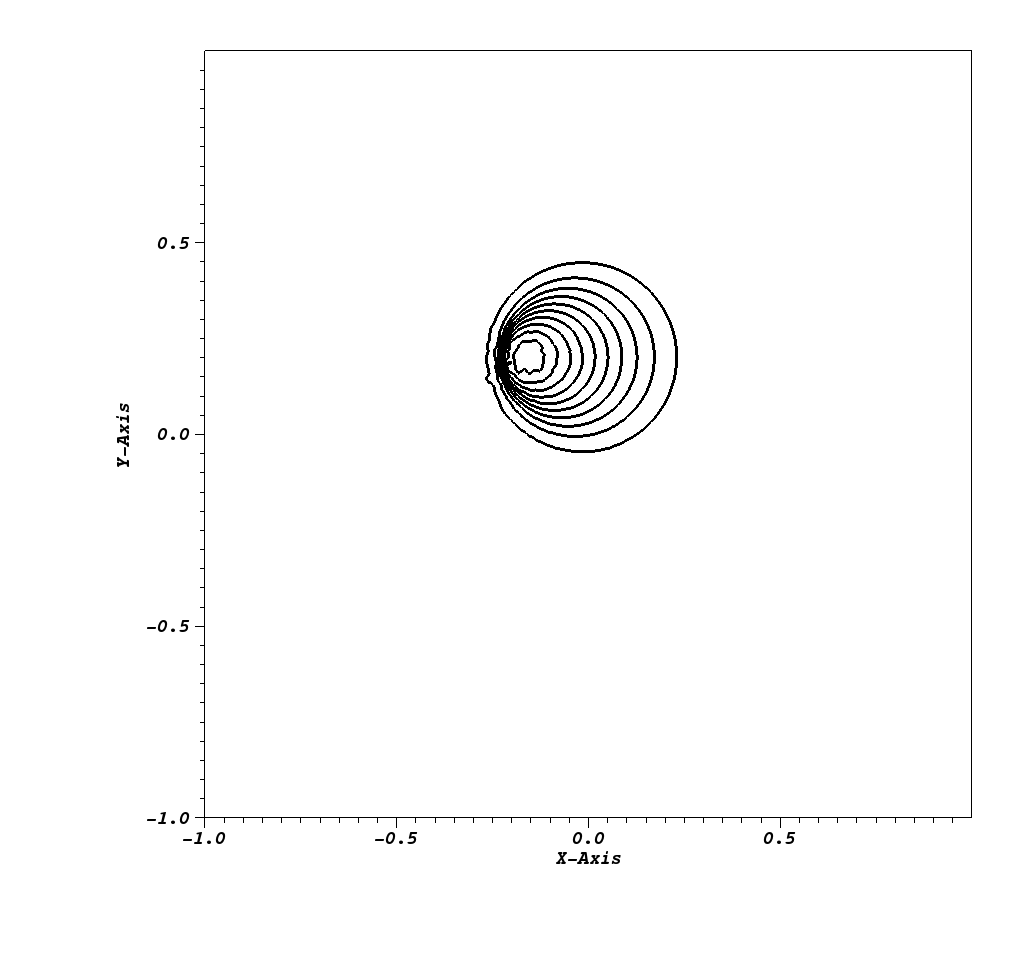}
    \caption{\label{fig:crash} 174 steps}
  \end{subfigure}%
     \caption{\label{Burgers_type_exact} 4-th order scheme in space and time}
 \end{figure}
 
\subsubsection*{Quadrature Formula}
Running the same experiments using a quadrature formula of order $N=5$ to evaluate $\bbF$
leads the results plotted in figure  \ref{Burgers_type_quadraturet} 
   \begin{figure}[!htp]
 \centering 
   \begin{subfigure}[b]{0.3\textwidth}
    \includegraphics[width=\textwidth]{Cos_B3_R13_Ref1_initial00000000.png}
    \caption{Initial Condition}
  \end{subfigure}%
     \begin{subfigure}[b]{0.3\textwidth}
    \includegraphics[width=\textwidth]{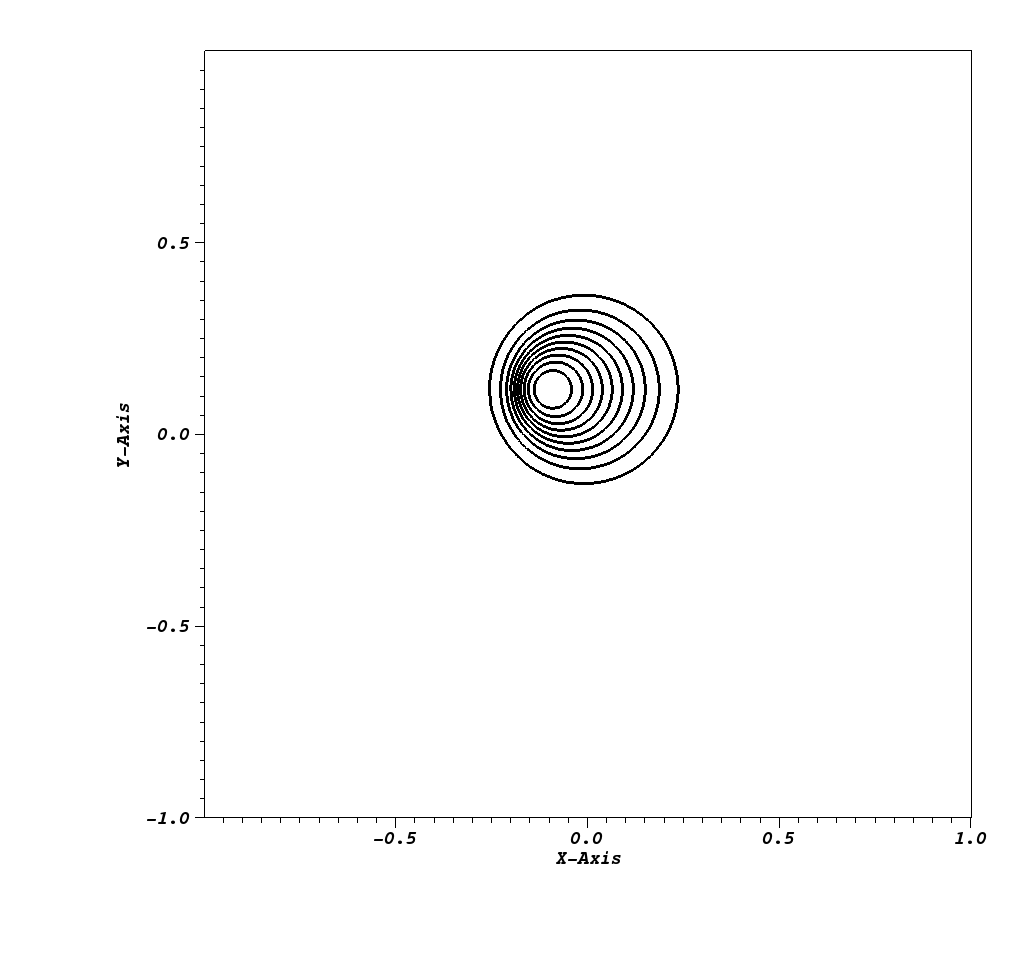}
    \caption{\label{fig:crash} 100 steps}
  \end{subfigure}%
       \begin{subfigure}[b]{0.3\textwidth}
    \includegraphics[width=\textwidth]{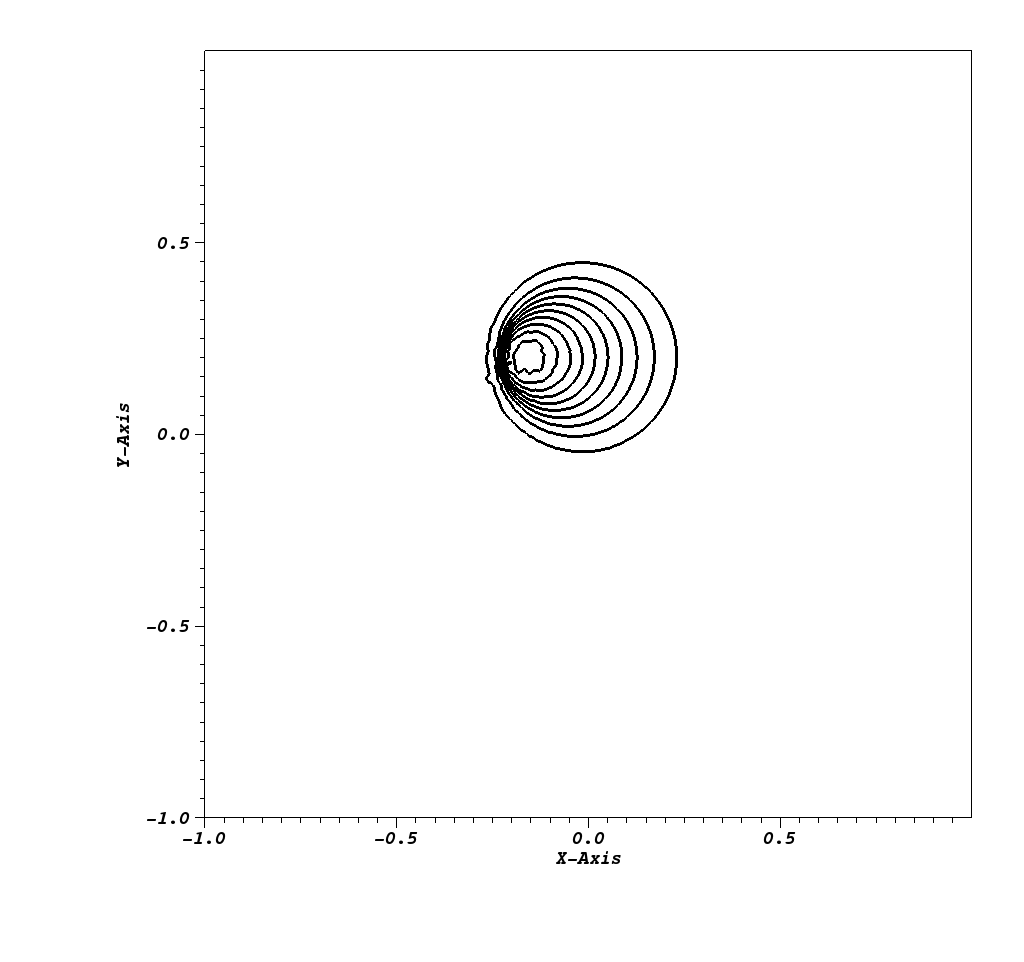}
    \caption{\label{fig:crash} 174 steps}
  \end{subfigure}%
     \caption{\label{Burgers_type_quadraturet} 4-th order scheme in space and time, CFL=0.1}
 \end{figure}
 Obviously, no difference can be seen and also the values are up to machine precision identically. 

\begin{remark}
We run this test with several different order for the Galerkin scheme 
(Lagrange/Bernstein basis) and also for the quadrature rule $(N=2-8)$ no major difference have be seen in the results. 
Therefore, we conclude that applying a quadrature formula to determine $\bbF$ have no effect on the accuracy and stability of the method. 
\end{remark}
All of our experiments verify that we obtain semi-discrete entropy stability for a pure Galerkin scheme 
if we apply the correction term together with our here derived boundary procedure. In the linear case, the entropy correction term 
is not needed but can be applied for stabilization reasons. However, our approach to estimate 
the boundary operator includes the linear cases used in the SBP-SAT framework  by setting the entropy variable $V=U$. Therefore, our approach can be seen as a natural extension to the linear case.

\section{Conclusion}

In this paper, we extend the investigation from \cite{abgrall2019analysis} 
and present a way to build entropy stable  Galerkin schemes.
By switching to the entropy variables, 
we generalize the SAT approach to nonlinear problems and develop new estimations for the boundary operators.
Our approach is in accordance with the boundary operators from the linear case, 
derived and used in the SBP-SAT community    (e.g. 
\cite{svard2014review, gassner2013skew, offner2019error} ).\\
By applying the boundary operator together with the correction term developed in \cite{abgrall2018general}
we are able to build semi-discrete entropy stable Galerkin schemes. 
Numerical experiments supports our theoretical analysis and demonstrate further that 
nearly no difference can be seen in the results if the boundary operator is developed using a quadrature rule or
exact integration. \\
Further research in this direction is the construction of fully entropy stable Galerkin schemes
using  SBP-SAT in time or relaxation Runge-Kutta methods.
 A more detailed analysis of the constructed boundary operator is also desirable. 
Finally, the influence of the entropy correction term for more advance problems 
should also be considered. In  \cite{ranocha2019reinterpretation}, 
a first study is already been done. However, this analysis has to be extended.

\section*{Acknowledgements}
P.\"O. has been funded by the the SNF grant (Number 200021\_175784) 
and by the UZH Postdoc grant.
This research was initiated by a first visit of JN at UZH,
and really started during ST postdoc at UZH. This postdoc was funded by an SNF grant 200021\_153604. The Los Alamos unlimited release number is LA-UR-19-32411.

\bibliographystyle{abbrv}
\bibliography{../literature}

\end{document}